\documentclass[11pt]{amsart}
\usepackage{mathrsfs}
\usepackage{amsfonts}
\usepackage{latexsym,amsmath,amssymb}
\usepackage[colorlinks=true,linkcolor=blue]{hyperref}
\usepackage{indentfirst}

 \renewcommand{\epsilon}{\varepsilon}

\newtheorem{theorem}{Theorem}[section]

 \newtheorem{lemma}[theorem]{Lemma}
 \newtheorem{remark}[theorem]{Remark}
 
 \newtheorem{problem}{Problem}
 \newtheorem{Corollary}[theorem]{Corollary}
 \newtheorem{proposition}[theorem]{proposition}
 
\newtheorem{deff}[theorem]{Definition}
 
 \newcommand{\bth}{\begin{theorem}}
 \newcommand{\ble}{\begin{lemma}}
 \newcommand{\bcor}{\begin{corr}}
 \newcommand{\bdeff}{\begin{deff}}
 \newcommand{\bprop}{\begin{proposition}}
 \newcommand{\ele}{\end{lemma}}
 \newcommand{\ecor}{\end{corr}}
 \newcommand{\edeff}{\end{deff}}
 
 \newcommand{\eprop}{\end{proposition}}

 \renewcommand{\Pi}{\varPi}

 \renewcommand{\epsilon}{\varepsilon}

  \newcommand{\R}{{\mathbb R}}

\numberwithin{equation}{section}

\thanks{The first author was supported by National Science Foundation of China(No.12101145) and Guangxi Science Technology Project (No. GuikeAD22035202).}

\title
[Nonlinear second boundary conditions ]{The constant mean curvature hypersurfaces
with prescribed gradient image}
\author{Rongli Huang}
\address{School of Mathematics and Statistics, Guangxi Normal University,
Guilin, Guangxi 541004, People's Republic of China,
 E-mail: ronglihuangmath@gxnu.edu.cn}
\author{Dayan Wei}
\address{School of Mathematics and Statistics, Guangxi Normal University,
Guilin, Guangxi 541004, People's Republic of China,
 E-mail: dayanweimath@gxnu.edu.cn}

\author{Yunhua Ye}
\address{School of Mathematics, Jiaying University,
Meizhou, Guangdong 514015, People's Republic of China,
E-mail: mathyhye@163.com}
\date{}

\begin{document}
\begin{abstract}
In this paper, we consider the existence of constant mean curvature hypersurfaces
 with prescribed gradient image.  Let $\Omega$ and $\tilde{\Omega}$ be uniformly convex
 bounded domains in $\mathbb{R}^n$ with smooth boundary. We show that there exists unique convex solutions  for
 the second boundary value problem of constant mean curvature equations.
\end{abstract}
\maketitle
%\section{}
%\subsection{}

\let\thefootnote\relax

\section{Introduction}
 Let $\mathbb{R}^{n,1}$ be the Minkowski space with the Lorentzian metric
 \begin{equation}\label{e1.1.1}
     ds^2=\sum^n_{i=1}dx_i^2-dx_{n+1}^2.
 \end{equation}

We consider convex spacelike hypersurfaces with constant mean curvature in Minkowski space $\mathbb{R}^{n,1}$.
Any such hypersurface can be written locally as a graph of a function $x_{n+1} = u(x)$, $x\in \mathbb{R}^n$, satisfying the spacelike condition
\begin{align}
    |Du|<1.
\end{align}

Here the constant mean curvature equation can be written as
\begin{equation}\label{e1.1.3}
    \mathrm{div}\left(\frac{Du}{\sqrt{1-|Du|^2}}\right)=c,  \quad x\in \Omega,
\end{equation}
in conjunction with the so-called second boundary value problem
\begin{equation}\label{e1.1.4}
    Du(\Omega)=\tilde{\Omega},
\end{equation}
where $c$ is a constant to be prescribed below,  and $\Omega$, $\tilde{\Omega}$ are uniformly convex bounded domains with smooth boundary in $\mathbb{R}^n$. The boundary condition $(\ref{e1.1.4})$ is natural for mean curvature equations of the form $(\ref{e1.1.3})$ because these equations are elliptic precisely on locally uniformly convex solutions, and the gradient map $Du$ is then a diffeomorphism of $\Omega$ onto its image $Du(\Omega) \subset {B}_1(0)$, where $B_1(0)$ is the unit ball in $\mathbb{R}^n$
with the Klein model of the hyperbolic geometry $\{(x,1)\in \mathbb{R}^{n,1}, |x|<1\}$.

The problem of fully nonlinear partial differential equations with second boundary value conditions have been studied  for a long time. Urbas \cite{Ju1}-\cite{Ju2} studied
the existence of globally smooth solutions to Monge-Amp\`{e}re type and a class of Hessian equations subject to the second boundary condition. Nessi-Gregory \cite{Nessi}
found the existence of globally smooth classical solutions of a new class of modified-Hessian equations, closely related to the Optimal Transportation Equation which satisfying the second boundary value problem. Shibing Chen etal \cite{Chen1} established the global $C^{2,\alpha}$ and $W^{2,p}$ regularity for the Monge-Amp\`{e}re equation subject to second boundary condition. Later, Savin etal \cite{Savin} studied the global regularity of $W^{2,1+\epsilon}$ estimates for Monge-Amp\`{e}re equation subject to second boundary condition and obtain the optimality of the exponent $1+\epsilon$.

In recent years, the curvature equations arise from geometry problems with second boundary value condition have aroused widespread interest among researchers.
Urbas established the existence of Weingarten hypersurfaces with prescribed gradient image in \cite{J2} and constructed a smooth pseudoconvex pair $(D_1,D_2)$
of domains in $\mathbb{R}^2$ with equal areas such that there is no globally smooth minimal Lagrangian diffeomorphism from $\overline{D_1}$ onto $\overline{D_2}$ in \cite{Ju3}.
Brendle and Warren \cite{SM} proved that there exists a diffeomorphism $f: \Omega \rightarrow$ $\tilde{\Omega}$ such that the graph $\Sigma=\{(x, f(x)): x \in \Omega\}\subset \mathbb{R}^n \times \mathbb{R}^n$ is a minimal Lagrangian submanifold  if $\Omega$ and $\tilde{\Omega}$ are uniformly convex. In \cite{Xin1,Xin2}, Xin  derived  Gauss curvature estimates for the $n$-graph $S \subset \mathbb{R}^{m+n}$ with prescribed mean curvature and proved the evolution equations of mean curvature flow have a long time smooth solution.
In a series works of \cite{HR2}-\cite{RS}, The first author and his coauthors considered the second boundary value problem for a class of Lagrangian mean curvature equation
by using elliptic or parabolic methods.

There is also a lot of literature on the study of spacelike hypersurfaces in Minkowski space.
Treibergs considered an entire spacelike hypersurface in Minkowski space may be globally represented as $x_{n+1}=f(z)$ with $z \in \mathbf{R}^n$ and the gradient of $f$ smaller than 1, obtained the set of entire constant mean curvature spacelike hypersurfaces may be identified with the set $Q$ of boundary cones in \cite{TAE}. He showed that for any $f\in C^2(\mathbb{S}^{n-1})$, there is a spacelike, convex, constant mean curvature hypersurface $\mathcal{M}_u = \{(x, u(x))|x\in \mathbb{R}^n \}$ with bounded principal curvatures, such that as $|x|\to\infty, u(x)\to |x|+ f\left(\frac{x}{|x|}\right)$.
After this, Treibergs and Choi proved the Gauss map of a spacelike constant mean curvature hypersurface of Minkowski space is a harmonic map to hyperbolic space in \cite{Choi}.
Then in \cite{Li}, Li extended Treibergs' results \cite{TAE} and proved the existence of constant Gauss curvature hypersurfaces with Gauss image a unit ball.
Later,  Bayard \cite{Bayard}  proved the existence of entire spacelike hypersurfaces of prescribed negative scalar curvature in Minkowski space, then the author and Schn$\ddot{u}$rer studied entire spacelike hypersurfaces of constant Gauss curvature in
Minkowski space in \cite{BaySch}. Aquino and Lima showed complete spacelike hypersurface immersed with constant mean curvature or bounded mean curvature must be totally umbilical in \cite{Aquino1}, then the authors and Bezerra studied the hypersurfaces with constant normalized scalar curvature $R$ immersed into the de Sitter space $\mathbb{S}_1^{n+1}$ in \cite{Aquino2}.
Recently, in  \cite{WX}, Wang and Xiao construct strictly convex, spacelike, constant $\sigma_k$ curvature hypersurface with bounded principal curvature, whose image of the Gauss map is the unit ball. Then Ren, Wang and Xiao showed exists a unique, entire, strictly convex, spacelike  hypersurface $\mathscr{M}_u$ satisfying prescribed asymptotic behavior at infinity in  \cite{RWX}. Many other researchers have studied surfaces of constant mean curvature from other perspectives. Huang \cite{HR1} considered the second boundary value problems for mean curvature flow. Based on the parabolic equation, he constructed the translating
solitons with prescribed Gauss image in Minkowski space. In Minkowski space, there have been fruitful results on the prescribed curvature problems for spacelike
entire hypersurfaces. Alexander I. Bobenko, Tim Hoffmann and Nina Smeenk \cite{BHS} define discrete constant mean curvature (CMC) surfaces in the three-dimensional Euclidean and Lorentz spaces in terms of sphere packings with orthogonally intersecting circles. Then they construct discrete CMC surfaces in $\mathbb{R}^3$  and investigate spacelike discrete CMC surfaces in the Lorentz space $\mathbb{R}^{2,1}$.

In this paper, we will investigate convex, spacelike
hypersurfaces of constant mean curvature equation with prescribed gradient image. Our main Theorems are stated as follows.
\begin{theorem}\label{t1.1}
 Suppose that $\Omega$, $\tilde{\Omega}$ are uniformly convex bounded domains with smooth boundary in $\mathbb{R}^n$ and  $\tilde{\Omega}\subset\subset B_1(0)$.
 Then there exists a uniformly convex solution $u\in C^{\infty}(\bar{\Omega})$ and a unique constant $c$ solving \eqref{e1.1.3} and \eqref{e1.1.4}.
 Here $u$ is unique up to a constant.
\end{theorem}
To obtain the existence result, we use the continuity method and by carrying out a priori estimates on the solutions to \eqref{e1.1.3} and \eqref{e1.1.4}.

We go on considering the  constant mean curvature equation  in $\mathbb{R}^{n+1}$:
\begin{equation}\label{e12a}
\mathrm{div}\left(\frac{Du}{\sqrt{1+|Du|^2}}\right)=c,  \quad x\in \Omega,
\end{equation}
associated with the second boundary value problem
\begin{equation}\label{e12b}
Du(\Omega)=\tilde{\Omega}.
\end{equation}

Based on the same proof as Theorem \ref{t1.1}, an immediate consequence of  the above problem  is the following:
\begin{theorem}\label{t1.2}
 Suppose that $\Omega$, $\tilde{\Omega}$ are uniformly convex and bounded domains with smooth boundary in $\mathbb{R}^n$.
  Then there exists a uniformly convex solution $u\in C^{\infty}(\bar{\Omega})$ and a unique constant $c$ solving \eqref{e12a} and \eqref{e12b}.
  Here $u$ is unique up to a constant.
\end{theorem}

\begin{remark}
Theorem \ref{t1.2} is a generalization of Urbas' works \cite{J2}, where he considered the Weingarten curvature equation $F[u]=g(x,u)$ with $g(x,z)\rightarrow \infty$
as $z\rightarrow \infty$ and $g_z>0$, which rules out the constant case.
\end{remark}

The rest of this article is organized as follows. In section 2, we introduce some basic formulas and notations, and then present the structure condition for the mean curvatrue equation. Thus in section 3, we devote to carry out the strictly oblique estimate and then in section 4 we obtain the  $C^2$ estimate according to the structure properties of the operators $G$ and $\tilde{G}$. We will give the special case of solution of $(\ref{e1.1.3})$ and  $(\ref{e1.1.4})$ and prove the main theorem by the continuity method as same as Brendle and Warren's work \cite{SM} in section 5 and 6.

\section{Preliminaries}
In this section, we will derive some basic formulas for the geometric quantities of spacelike hypersurfaces in Minkowski space $\mathbb{R}^{n,1}$.
We then give the  structure condition for the mean curvature equation referring to \cite{WHB1}.

We start with the definitions and notations of differential geometry for graphic hypersurface in Minkowski space $\mathbb{R}^{n,1}$,
the readers can see \cite{EH} and \cite{Ec} for a nice introduction.
A spacelike hypersurface $M\subset \mathbb{R}^{n,1}$ is a codimension one submanifold  with the Lorentzian metric \eqref{e1.1.1}
whose induced metric is Riemannian. Locally $M$ can be written as a graph
\begin{align}
    \mathcal{M}_u = \{X = (x, u(x))|x \in \mathbb{R}^n\}
\end{align}
satisfying the spacelike condition $(\ref{e1.1.4})$.
It is easy to see that the induced metric and second
fundamental form of $M $ are given by
\begin{align}
g_{ij}=\delta_{ij}-D_{i}uD_{j}u,\quad 1\leq i,j\leq n.
\end{align}

While the inverse of the induced metric and second fundamental form of $M$ are given respectively by
\begin{align}
g^{ij}=\delta_{ij}+\frac{D_{i}uD_{j}u}{1-|Du|^{2}},\quad 1\leq i,j\leq n,
\end{align}
and
\begin{align}
h_{ij}=\frac{D_{ij}u}{\sqrt{1-|Du|^{2}}},\quad 1\leq i,j\leq n.
\end{align}
The timelike unit normal vector field to $M$ is expressed by
\begin{align}
\nu=\frac{(Du,1)}{\sqrt{1-|Du|^{2}}}.
\end{align}
Specially, the mean curvature of $M$ is written as
\begin{equation}\label{e1.11}
H=\sum_{1\leq i\leq n}\kappa_{i}.
\end{equation}

Let $B_1(0)$ be the unit ball in $\mathbb{R}^n$ with the Klein model of the hyperbolic geometry $\{(x,1)\in \mathbb{R}^{n,1}, |x|<1\}$.
Following Lemma 4.5 in \cite{Choi}, the Gauss map of the graph $(x,u(x))$ is described from $R^n$ to $B_1(0)$ as
$$ \mathbb{G} : x \mapsto Du(x).$$

We aim to construct convex spacelike constant mean curvature hypersurfaces with prescribed Gauss
image over any strictly convex domains by solving problem \eqref{e1.1.3} and \eqref{e1.1.4}.

We assume that $u\in C^{2}(\Omega)$ for some domain in $\mathbb{R}^{n}$  .
For $1\leq i,j,k \leq n$, we denote $$D_{i}u=\dfrac{\partial u}{\partial x_{i}},
D_{ij}u=\dfrac{\partial^{2}u}{\partial x_{i}\partial x_{j}},
D_{ijk}u=\dfrac{\partial^{3}u}{\partial x_{i}\partial x_{j}\partial
x_{k}}, \cdots $$
and $|Du|=\sqrt{\sum_{i=1}^{n}|D_{i}u|^{2}}.$
It follows from \cite{LLJ}  that  we can state various geometric quantities associated with the graph of  $u\in C^{2}(\Omega)$.
In the coordinate systems, Latin indices range from 1 to $n$ and indicate quantities in the graph.
We adopt the Einstein's convention of summation over repeated indices in the following.\\
Let
\begin{equation}
        F(\kappa_{1},\cdots, \kappa_{n}):=\sum_{i=1}^n\kappa_{i},
\end{equation}
be a smooth function on the positive cone
\begin{align*}
    \Gamma^+_n:=\left\{(\kappa_1,\cdots,\kappa_n)\in \mathbb{R}^n:\kappa_i>0,\ i=1,\cdots,n\right\}.
\end{align*}

  The $F$ is a smooth symmetric function defined on $\Gamma^+_n$. Acorrding to  \cite{ou}, the $F$ satisfies
  \begin{equation}
\sum_{i=1}^n\frac{\partial F}{\partial \kappa_i}\kappa_i= F,
\end{equation}
\begin{equation}\label{e2.2.10}
\frac{\partial F}{\partial \kappa_i}>0,\ \ 1\leq i\leq n\ \  \text{on}\ \  \Gamma^+_n,
\end{equation}
\begin{equation}
\sum_{i=1}^n\frac{\partial F}{\partial \kappa_i}=n \ \   \text{on}\ \  \Gamma^+_n,
\end{equation}
and
\begin{equation}\label{e2.2.12}
\left(\frac{\partial^2 F}{\partial \kappa_i\partial \kappa_j}\right)\leq 0\ \  \text{on}\ \  \bar{\Gamma^+_n}.
\end{equation}
\begin{lemma}\label{L2.1}
    Assume that $\Omega, \tilde{\Omega}$ are bounded, uniformly convex domains with smooth boundary in $\mathbb{R}^n$ and  $\tilde{\Omega}\subset\subset B_1(0)$. If the strictly convex solution to \eqref{e1.1.3} and \eqref{e1.1.4} exists. Then there exist positive constants $\Lambda_1$ and $\Lambda_2$, depending only on $\Omega$ and $\tilde{\Omega}$, such that there holds
\begin{align}\label{e2.2.8}
    \Lambda_1 \leq F(\kappa) \leq \Lambda_2.
\end{align}
\end{lemma}
\begin{proof}
Since $Du(\Omega)=\tilde{\Omega}$, $\tilde{\Omega} \subset\subset {B}_1(0)$,  we have
\begin{align*}
|\tilde{\Omega}|&=\int_{\tilde{\Omega}}dy=\int_{\Omega}\det D^2udx\\&=\int_{\Omega}\left(1-|Du|^2\right)^{\frac{n+2}{2}}\frac{\det D^2u}{\left(1-|Du|^2\right)^{\frac{n+2}{2}}}dx.
\end{align*}
By noting that
\begin{align*}
\kappa_1\cdot\kappa_2\cdots\kappa_n=\frac{\det D^2u}{\left(1-|Du|^2\right)^{\frac{n+2}{2}}}\quad \text{and}\quad\max_{\overline{\Omega}}\left(1-|Du|^2\right)^\frac{n+2}{2}=\max_{\tilde{\Omega}}\left(1-|y|^2\right)^\frac{n+2}{2}\leq 1,
\end{align*}
then we can get
\begin{align*}
    |\tilde{\Omega}|\leq \int_{\Omega}\kappa_1\cdot\kappa_2\cdots\kappa_ndx\leq\left(\frac{\kappa_1+\cdots+\kappa_n}{n}\right)^n|\Omega|=\frac{[F(\kappa)]^n}{n^n}|\Omega|.
\end{align*}
Then we arrive at
\begin{align*}
    F(\kappa)\geq n\left(\frac{|\tilde\Omega|}{|\Omega|}\right)^{\frac{1}{n}},
\end{align*}
Integrating  over ${\Omega}$ on the both sides of \eqref{e1.1.3}, we can obtain
\begin{align*}
   \int_\Omega \mathrm{div}\left(\frac{Du}{\sqrt{1-|Du|^2}}\right) dx = c\int_\Omega dx,
\end{align*}
 where $\nu=\left(\nu_1, \nu_2, \cdots, \nu_n\right) $ be the unit inward normal vector of $\partial \Omega $.
 By using the divergence theorem, we see that
\begin{align*}
    c=\frac{1}{|\Omega|}\int_{\partial\Omega}\frac{Du\cdot\nu}{\sqrt{1-|Du|^2}}ds,
\end{align*}
and then
\begin{align*}
    F(\kappa)|\Omega|=c|\Omega| \leq \int_{\partial \Omega} \frac{|D u \cdot \nu|}{\sqrt{1-|D u|^2}} d s\leq \int_{\partial \Omega} \frac{|D u |}{\sqrt{1-|D u|^2}} d s.
\end{align*}
 Let
\begin{align*}
\Lambda_1=n\left(\frac{|\tilde\Omega|}{|\Omega|}\right)^{\frac{1}{n}}, \quad \Lambda_2= \frac{|\partial \Omega|}{|\Omega|} \max _{y \in \partial \tilde{\Omega}} \frac{|y|}{\sqrt{1-|y|^2}}.
\end{align*}
Thus the proof of \eqref{e2.2.8} is
completed.
\end{proof}
\begin{lemma}\label{e3.2.2} Assume that $\Omega$, $\tilde{\Omega}$ are bounded, uniformly convex domains with smooth boundary in $\mathbb{R}^{n}$ and   $\tilde{\Omega}\subset\subset B_1(0)$, $u$  is the strictly convex solution to $(\ref{e1.1.3})$ and $(\ref{e1.1.4})$. Then there exist positive constants $\Lambda_3$ and $\Lambda_4$,  such that there holds
  \begin{align}
 \label{e3.1}
 \Lambda_3\leq\sum^{n}_{i=1}\frac{\partial F}{\partial \kappa_{i}}\leq \Lambda_4
 \end{align}
and
\begin{equation}\label{e3.2}
\Lambda_3\leq\sum^{n}_{i=1}\frac{\partial F}{\partial \kappa_{i}}\kappa^{2}_{i}\leq \Lambda_4.
\end{equation}
\end{lemma}
\begin{proof}
  We observe that the operator
 $$F=\sum_{i=1}^n\kappa_{i}=c$$ and $$\frac{\partial F}{\partial \kappa_{i}}=1.$$
Then $$\sum^{n}_{i=1}\frac{\partial F}{\partial \kappa_{i}}=n,$$
with
$$\sum^{n}_{i=1}\frac{\partial F}{\partial \kappa_{i}}\kappa^{2}_{i}=\sum^{n}_{i=1}\kappa^{2}_{i},$$
and using Cauchy inequality, we obtain
$$ \frac{1}{n}(\kappa_1+\cdots+\kappa_n)^2\leq \kappa_1^2+\cdots+\kappa_n^2\leq (\kappa_1+\cdots+\kappa_n)^2 ,$$
combining with $(\ref{e2.2.8})$ in Lemma \ref{L2.1}, we know $F$ is bounded and
we obtain the desired results.
\end{proof}
The principal curvatures of $M\subset \mathbb{R}^{n,1}$ are the eigenvalues of the second fundamental form
$h_{ij}$ relative to $g_{ij}$, i.e., the eigenvalues of the mixed tensor $h^{j}_{i}\equiv h_{ik}g^{kj}$.
By \cite{LLJ} we remark that they are the eigenvalues of the symmetric matrix
\begin{equation}\label{e2.2.13}
a_{ij}=\frac{1}{v}b^{ik}D_{kl}ub^{lj},
\end{equation}
where $v=\sqrt{1-|Du|^{2}}$ and $b^{ij}$ is the positive square root of $g^{ij}$ taking the form
\begin{align*}
    b^{ij}=\delta_{ij}+\frac{D_{i}uD_{j}u}{v(1+v)}.
\end{align*}

The inverse of $b^{ij}$ is
\begin{align*}
    b_{ij}=\delta_{ij} - \frac{D_{i}uD_{j}u}{1+v},
\end{align*}
which is the square root of $g_{ij}$.\\
And then
\begin{align*}
    D_{ij}u=vb_{ik}a_{kl}b_{lj}.
\end{align*}

By some calculation, it yields
\begin{align*}
 a_{ij}=\frac{1}{v}\left (D_{ij}u - \frac{D_iuD_luD_{jl}u}{v(1+v)} - \frac{D_juD_luD_{il}u}{v(1+v)} + \frac{D_iuD_kuD_luD_juD_{kl}u}{v^2(1+v)^2}\right).
\end{align*}

Denote $\mathcal{A}=[a_{ij}]$ and $F[\mathcal{A}]=\sum_{i=1}^n\kappa_{i}$, where $(\kappa_1,\cdots, \kappa_n)$ are the  eigenvalues of
the symmetric matrix $[a_{ij}]$. Then the properties of the operator $F$ are reflected in (\ref{e2.2.8})-(\ref{e3.2}).
It follows from (\ref{e2.2.10}) that
\begin{equation*}
F_{ij}[\mathcal{A}]\xi_{i}\xi_{j}>0 \quad \mathrm{for}\quad \mathrm{all}\quad \xi\in \mathbb{R}^{n}-\{0\},
\end{equation*}
where
\begin{equation*}
F_{ij}[\mathcal{A}]=\frac{\partial F[\mathcal{A}]}{\partial a_{ij}}.
\end{equation*}

From \cite{JS} we see that $[F_{ij}]$ is diagonal if $\mathcal{\mathcal{A}}$ is diagonal, and in this case
\begin{equation*}
[F_{ij}]=\mathrm{diag}(\frac{\partial F}{\partial \kappa_{1}},\cdots, \frac{\partial F}{\partial \kappa_{n}})=\mathrm{diag}(1,\cdots,1).
\end{equation*}

If $u$ is convex, by (\ref{e2.2.13}) we deduce that the  eigenvalues of the  matrix $[a_{ij}]$ must be in $\bar{\Gamma^+_n}$.
Then (\ref{e2.2.12}) implies  that
\begin{align*}
F_{ij,kl}[\mathcal{A}]\eta_{ij}\eta_{kl}\leq 0,
\end{align*}
for any real symmetric matrix $[\eta_{ij}]$, where
\begin{equation*}
F_{ij,kl}[\mathcal{A}]=\frac{\partial^{2}F[\mathcal{A}]}{\partial a_{ij}\partial a_{kl}}.
\end{equation*}

According to the equation (\ref{e1.1.3}), we consider the nonlinear differential operators of the type
\begin{equation*}\label{e2.10}
G(Du,D^{2}u)=0.
\end{equation*}

As in \cite{J2}, by differentiating this equation once, we have
\begin{equation*}
G_{ij}D_{ijk}u+G_{i}D_{ik}u=0,
\end{equation*}
where we use the notation
\begin{equation*}
G_{ij}=\frac{\partial G}{\partial r_{ij}} \quad \text{and} \quad G_{i}=\frac{\partial G}{\partial p_{i}},
\end{equation*}
with $r$ and $p$  representing for the second derivatives and gradient variables respectively.
So as to prove the strict obliqueness estimate for the problem (\ref{e1.1.3})-(\ref{e1.1.4}), we need to recall some expressions from \cite{J2} for the derivatives of $G$.  We have
\begin{equation}\label{e2.2.14}
G_{ij}=F_{kl}\frac{\partial a_{kl}}{\partial r_{ij}}=\frac{1}{v}b^{ik}F_{kl}b^{lj}
\end{equation}
and
\begin{equation*}\label{e2.12}
G_{i}=F_{kl}\frac{\partial a_{kl}}{\partial p_i}=F_{kl}\frac{\partial }{\partial p_i}\left(\frac{1}{v}b^{kp}b^{ql}\right)u_{pq}.
\end{equation*}

A simple calculation yields
\begin{align*}
    G_i=\frac{u_i}{v^2}F_{kl}a_{kl}+\frac{2}{v}F_{kl}a_{ml}b^{ik}u_m.
\end{align*}

We observe that $\mathcal{T}_{G}=\sum^{n}_{i=1}G_{ii}$ is the trace of a product of three matrices by (\ref{e2.2.14}), so it is invariant under orthogonal transformations. Hence, to compute $\mathcal{T}_{G}$, we may assume for now that $[a_{ij}]$ is diagonal. By virtue of (\ref{e1.1.4}) and $\tilde{\Omega}\subset\subset B_{1}(0)$, we obtain that $Du$ and $\frac{1}{v}$  are bounded. Then eigenvalues of $[b^{ij}]$ are bounded between two controlled positive constants.
We can observe that
\begin{align*}
\mathcal{T}_{G}=\sum_{i=1}^nG_{ii}=\frac{1}{v}b^{ik}F_{kl}b^{li},
\end{align*}
it follows that there exist positive constants $\sigma_{1}$,
$\sigma_{2}$ depending only on the least upper bound of $|Du|$ in the set $\Omega$, such that
\begin{equation}\label{e2.14}
\sigma_{1}\mathcal{T}\leq\mathcal{T}_{G}\leq\sigma_{2}\mathcal{T},
\end{equation}
where $\mathcal{T}=\sum^{n}_{i=1}F_{ii}$.
By the concavity of $F$ and the positive definiteness of $[F_{ij}a_{ij}]$ imply that$[F_{ij}a_{ij}]$ is controlled by $F$, i.e.,
\begin{equation*}\label{e2.15}
0<F_{ij}a_{ij}=\sum^{n}_{i=1}F_{i}\kappa_{i}\leq F(\kappa_{1},\ldots,\kappa_{n}).
\end{equation*}
Thus $G_i$ is bounded.

Next, we will use the Lengendre transformation of $u$ which is the convex function $\tilde{u}$ on $\Tilde{\Omega}=Du(\Omega)$ define by \\
\begin{align*}
    \tilde{u}(y)=x\cdot Du(x)-u(x),
\end{align*}
and
\begin{align*}
    y=Du(x).
\end{align*}
It follows that
\begin{align*}
    \frac{\partial \tilde{u}}{\partial y_{i}}=x_{i},\quad\frac{\partial ^{2}\tilde{u}}{\partial y_{i}\partial y_{j}}=u^{ij}(x),
\end{align*}
where $[u^{ij}]=[D^2 u]^{-1}$. Then $\tilde{u}$ satisfies
\begin{align}\label{transfoemation G}
    \tilde{G}(y,D^2\tilde{u})=-G(y,[D^2\tilde{u}]^{-1})=-c \qquad \text{in} \quad \tilde{\Omega},
\end{align}
with the boundary condition
\begin{align*}
    D\tilde{u}(\tilde{\Omega})=\Omega.
\end{align*}
Moreover, $(\ref{transfoemation G})$ can be written as
\begin{align*}
    \tilde{F}[\tilde{a}_{ij}]=-c,
\end{align*}
where
\begin{align*}
    \tilde{F}[\tilde{a}_{ij}]=\tilde{F}(\eta_1,\eta_2,\cdots,\eta_n)=-F(\eta_1^{-1},\eta_2^{-1},\cdots,\eta_n^{-1}).
\end{align*}

Here $\tilde{F}$ satisfies the structural conditions of Lemma \ref{e3.2.2}, $\eta_1,\eta_2,\cdots,\eta_n$ are the eigenvalues of the matrix [$\Tilde{a}_{ij}$] and
\begin{align*}
    \tilde{v}=\sqrt{1-|y|^2},
\end{align*}
\begin{align*}
    [\tilde{a}_{ij}]=\sqrt{1-|y|^2}\tilde{b}_{ik}D_{kl}\tilde{u}\tilde{b}_{lj},
\end{align*}
\begin{align*}
    \tilde{b}_{ij}=\delta_{ij}-\frac{y_iy_j}{1+\tilde{v}},
\end{align*}
where $[\tilde{b}^{ij}]$ is denoted the inverse matrix of $[\tilde{b}_{ij}]$ , it is given by
\begin{align*}
    \tilde{b}^{ij}=\delta_{ij}+\frac{y_iy_j}{\tilde{v}(1+\tilde{v})}.
\end{align*}

Since $y\in\tilde{\Omega}$, the eigenvalues of $[\tilde{b}_{ij}]$ and $[\tilde{b}^{ij}]$ are bounded between two controlled positive constants $\sigma_3,\sigma_4$. Consequently, we have
\begin{align}\label{e2.2.19}
\sigma_3\tilde{\mathcal{T}}\leq\tilde{\mathcal{{T}}}_{\tilde{G}}\leq \sigma_4\tilde{\mathcal{T}},
\end{align}
where $\tilde{\mathcal{T}}=\sum_{i=1}^n\tilde{F}_{ii}$, $\Tilde{\mathcal{T}}_{\tilde{G}}=\sum_{i=1}^n\Tilde{G}_{ii}$, we can conclude that
\begin{align*}
\tilde{G}_{ij}D_{ijk}\tilde{u}+\tilde{G}_{y_k}=0,
\end{align*}
where
\begin{align*}
    \tilde{G}_{ij}=\sqrt{1-|y|^2}\tilde{b}_{ik}\tilde{F}_{kl}\tilde{b}_{lj},
\end{align*}
and
\begin{align*}    \tilde{G}_{y_i}&=\tilde{F_{kl}}\frac{\partial}{\partial y_i}\left(\sqrt{1-|y|^2}\tilde{b}_{kp}\tilde{b}_{ql}\right)\tilde{u}_{pq}\\
&=-\frac{y_i}{1-|y|^2}\tilde{F}_{kl}\tilde{a}_{kl}+2\tilde{F}_{kl}\tilde{a}_{lm}\frac{\partial}{\partial y_i}(\tilde{b}_{kp})\cdot \tilde{b}^{pm}\\
 &=-\frac{y_i}{1-|y|^2}\tilde{F}_{kl}\tilde{a}_{kl}+\tilde{F}_{kl}\tilde{a}_{lm}\left(-\frac{\delta_{ik}y_p}{1+\tilde{v}}-\frac{y_{k}\delta_{ip}}{1+\tilde{v}}-\frac{y_ky_i}{(1+\tilde{v})^2\tilde{v}}\right)\cdot \left(\delta_{pm}+\frac{y_py_m}{\tilde{v}(1+\tilde{v})}\right)\\
 &=-\frac{y_i}{1-|y|^2}\tilde{F}_{kl}\tilde{a}_{kl}+\tilde{F}_{kl}\tilde{a}_{lm}c_{ikmp},
\end{align*}
$c_{ikmp}$ depending on $y,y_i,y_k,y_m,y_p$ and $c$, $c$ is a constant. We can similarly show that  $\tilde{G}_{y_i}$ is bounded. Therefore, there holds
  \begin{align}
 |\tilde{G}_{y_i}|\leq \Lambda_5,
  \end{align}
 where $\Lambda_5$ is a uniformly positive constant.
\section{The strict obliqueness estimate}
In this section, we will give the structural conditions for the operator $G$.
We will carry out the strict obliqueness estimates.
\begin{Corollary}\label{c3.3}
 Assume that $\Omega$, $\tilde{\Omega}$ are bounded, uniformly convex domains with smooth boundary in $\mathbb{R}^{n}$ and  $\tilde{\Omega}\subset\subset B_{1}(0)$. Suppose that $u \in C^2(\overline{\Omega})$ is a uniformly convex solution of \eqref{e1.1.3} and  \eqref{e1.1.4}. If the strictly convex solution to \eqref{e1.1.3} and  \eqref{e1.1.4}.
There   exists  uniformly  positive constants  $\Lambda_6,\Lambda_7$,  depending  on the known data,  such that there holds
\begin{equation}\label{e3.4}
\Lambda_6\leq \sum^{n}_{i=1}\frac{\partial G}{\partial \lambda_{i}}\leq \Lambda_7,
\end{equation}
\begin{equation}\label{e3.5}
\Lambda_6\leq \sum^{n}_{i=1}\frac{\partial G}{\partial \lambda_{i}}\lambda^{2}_{i}\leq \Lambda_7,
\end{equation}
where $\lambda_1,\cdots, \lambda_n$ are the eigenvalues of Hessian matrix $D^2 u$ at $x \in \Omega$.
\end{Corollary}
 Gathering the results obtained above we arrive at the following structural conditions for the operator $\tilde{G}$.
\begin{Corollary}\label{c3.4}
 Assume that $\Omega$, $\tilde{\Omega}$ are bounded, uniformly convex domains with smooth boundary in $\mathbb{R}^{n}$
 and  $\tilde{\Omega}\subset\subset B_{1}(0)$. Suppose that $u \in C^2(\overline{\Omega})$ is a uniformly convex solution of \eqref{e1.1.3} and  \eqref{e1.1.4}.
Then there exists uniformly  positive constants  $\Lambda_6,\Lambda_7$ which depending on the known data, such that there holds
\begin{equation}\label{e3..3.3}
\Lambda_6\leq \sum^{n}_{i=1}\frac{\partial \tilde{G}}{\partial \mu_{i}}\leq \Lambda_7,
\end{equation}
\begin{equation}\label{e3..3.4}
\Lambda_6\leq \sum^{n}_{i=1}\frac{\partial \tilde{G}}{\partial \mu_{i}}\mu^{2}_{i}\leq \Lambda_7,
\end{equation}
where $\mu_1,\cdots, \mu_n$ are the eigenvalues of the Hessian matrix $D^2 \tilde{u}$ at $x \in \Omega$.
\end{Corollary}

From \cite{Ju1} and Proposition A.1 in \cite{SM}, we give the following
\begin{deff}
A smooth function $h:\mathbb{R}^n\rightarrow\mathbb{R}$ is called the defining function of $\tilde{\Omega}$,
 if
$$\tilde{\Omega}=\{p\in\mathbb{R}^{n} : h(p)>0\},\quad |Dh|_{{\partial\tilde{\Omega}}}=1,$$
and there exists $\theta>0$ such that for any $p=(p_{1},\cdots, p_{n})\in \tilde{\Omega}$ and $\xi=(\xi_{1}, \cdots, \xi_{n})\in \mathbb{R}^{n}$,
$$\frac{\partial^{2}h}{\partial p_{i}\partial p_{j}}\xi_{i}\xi_{j}\leq -\theta|\xi|^{2}.$$
\end{deff}
Therefore, the diffeomorphism condition $Du(\Omega) =\tilde\Omega$ in (\ref{e1.1.4}) is equivalent to
\begin{align}
    h(Du)=0,\quad x\in\Omega.
\end{align}

And then the mean curvature equation (\ref{e1.1.3})-(\ref{e1.1.4}) is equivalent to the following elliptic problem
\begin{equation}\label{e3.3.4}
\begin{cases}
G(Du,D^2u)=c, \quad x\in \Omega,\\
\qquad\;\;  h(Du)=0,\quad x\in \partial\Omega.
 \end{cases}
\end{equation}
\begin{lemma}$($See Lemma 3.4 in \cite{RS}$)$\label{l3.30}
Assume that  $[A_{ij}]$ is semi-positive real symmetric matrix and $[B_{ij}]$, $[C_{ij}]$ are two real symmetric matrices. Then
$$2A_{ij}B_{jk}C_{ki}\leq A_{ij}B_{ik}B_{jk}+A_{ij}C_{ik}C_{jk}.$$
\end{lemma}
According to the proof in \cite{Ju2}, we can verify the oblique boundary condition.
\begin{lemma}\label{lem3.5}$($See Lemma 3.1 in \cite{OC}$)$
    If $u$ is a  smooth uniformly convex solution of $\eqref{e1.1.3}$ and $\eqref{e1.1.4}$ , and then the boundary condition is oblique, i.e.,
    \begin{align}\label{oblique equation}
        \langle \nu(x),\tilde{\nu}(Du(x))\rangle\geq0,\quad x\in\partial \Omega,
    \end{align}
    where $\nu$ and $\tilde{\nu}$ denote the inner unit normals of $\Omega$ and $\tilde{\Omega}$.
\end{lemma}
For the convenience, we denote $\beta=(\beta^{1}, \cdots, \beta^{n})$ with $\beta^{i}:=h_{p_{i}}(Du)$, and $\nu=(\nu_{1},\cdots,\nu_{n})$ as the unit inward normal vector at $x\in\partial\Omega$. The expression of the inner product is
\begin{equation*}
\langle\beta, \nu\rangle=\beta^{i}\nu_{i}.
\end{equation*}
\begin{lemma}\label{llll3.4}
  Assume that $\Omega$, $\tilde{\Omega}$ are bounded, uniformly convex domains with smooth boundary in $\mathbb{R}^{n}$ and  $\tilde{\Omega}\subset\subset B_1(0)$ . If $u$ is a strictly convex solution to $(\ref{e3.3.4})$, then the strict obliqueness estimate
\begin{equation}\label{e3.3.12}
\langle\beta, \nu\rangle\geq \frac{1}{C_1}>0
\end{equation}
holds on $\partial \Omega$ for some universal constant $C_1$, which depends only on $\Omega,$ $\tilde{\Omega}$.
\end{lemma}
\begin{proof}
The proof of this is similar in the parts to the proof of \cite{J2}, but it is different from the elliptic operators and the barrier functions which are given in this paper. It follows from the maximum principle according to the structural conditions of the operator $G$, and is proved in the same way as \cite{LL1}.
Define
$$\omega=\langle \beta,\nu\rangle+\tau h(Du),$$
where $\tau$ is a positive constant to be determined. Let $x_0\in \partial\Omega$  such that
$$\langle \beta,\nu\rangle(x_0)=h_{p_k}(Du(x_0))\nu_k(x_0)=\min_{\partial\Omega}\langle \beta,\nu\rangle.$$

By rotation, we may assume that $\nu(x_0)=(0,\cdots,0,1)$. Applying the above assumptions and the boundary condition, we find that
$$\omega(x_0)=\min_{\partial\Omega} \omega=h_{p_{n}}(Du(x_0)).$$

By the smoothness of $\Omega$ and its convexity, we extend $\nu$ smoothly to a tubular neighborhood of $\partial\Omega$ such that in the matrix sense
\begin{equation}\label{e3.3.9}
  \left(\nu_{kl}\right):=\left(D_k\nu_l\right)\leq -\frac{1}{C_{2}}\operatorname{diag} (\underbrace {1,\cdots, 1}_{n-1},0),
\end{equation}
where $C_{2}$ is a positive constant. By Lemma \ref{lem3.5}, we see that $h_{p_{n}}(Du(x_0))\geq0$.

At the minimum point  $x_0$, it yields
\begin{equation}\label{e3.3.14}
 0=\omega_r=h_{p_np_k}u_{kr}+h_{p_k}\nu_{kr}+\tau h_{p_k}u_{kr}, \quad 1\leq r\leq n-1.
\end{equation}
We assume that the following key result
\begin{equation}\label{e3.1.1}
\omega_n(x_0)>-C_{3}
\end{equation}
holds which will be proved later, where $C_{3}$ is a positive constant depending only on $\Omega$ and $\tilde{\Omega}$. We observe that \eqref{e3.1.1} can be rewritten as
\begin{equation}\label{e3.7}
h_{p_np_k}u_{kn}+h_{p_k}\nu_{kn}+ \tau h_{p_k}u_{kn}>-C_{3}.
\end{equation}

Multiplying $h_{p_n}$ on both sides of \eqref{e3.7} and $h_{p_r}$ on both sides of \eqref{e3.3.14} respectively, and summing up together, one gets
\begin{align*}
 \tau h_{p_k}h_{p_l}u_{kl}\geq -C_{3}h_{p_n}- h_{p_k}h_{p_l}\nu_{kl}- h_{p_k}h_{p_np_l}u_{kl}.
\end{align*}

Combining (\ref{e3.3.9}) with
$$ 1\leq r\leq n-1,\quad h_{p_k}u_{kr}=\frac{\partial h(Du)}{\partial x_r}=0,\quad h_{p_k}u_{kn}=\frac{\partial h(Du)}{\partial x_n}\geq 0,\quad -h_{p_np_n}\geq 0,$$
we have
$$\tau h_{p_k}h_{p_l}u_{kl}\geq-C_{3}h_{p_n}+\frac{1}{C_{2}}|Dh|^2-\frac{1}{C_{2}}h^2_{p_n}
\geq-C_{4}h_{p_n}+\frac{1}{C_{4}}-\frac{1}{C_{4}}h^2_{p_n},$$
where $C_{4}=\max\{C_{2},C_{3}\}$.
Now, to obtain the estimate $\langle \beta,\nu\rangle$, we divide $-C_{4}h_{p_n}+\frac{1}{C_{4}}-\frac{1}{C_{4}}h^2_{p_n}$ into two cases at $x_0$.

Case (i).  If
$$-C_{4}h_{p_n}+\frac{1}{C_{4}}-\frac{1}{C_{4}}h^2_{p_n}\leq \frac{1}{2C_{4}},$$
then
$$h_{p_k}(Du)\nu_{k}=h_{p_n}\geq \sqrt{\frac{1}{2}+\frac{C^{4}_{4}}{4}}-\frac{C^{2}_{4}}{2}.$$
It means that there is a uniform positive lower bound for $\underset{\partial\Omega}\min \langle \beta,\nu\rangle$.

Case (ii). If
$$-C_{4}h_{p_n}+\frac{1}{C_{4}}-\frac{1}{C_{4}}h^2_{p_n}> \frac{1}{2C_{4}},$$
then we know that there is a positive lower bound for $h_{p_k}h_{p_l}u_{kl}$.

The unit inward normal vector of $\partial\Omega$ can be expressed by $\nu=D\tilde{h}$. For the same reason, $\tilde{\nu}=Dh$, where $\tilde{\nu}=(\tilde{\nu}_{1}, \tilde{\nu}_{2},\cdots,\tilde{\nu}_{n})$ is the unit inward normal vector of $\partial\tilde{\Omega}$.  $\tilde{h}$ is the defining function of $\Omega$ . That is,
$$\Omega=\{\tilde{p}\in\mathbb{R}^{n} : \tilde{h}(\tilde{p})>0\},\ \ \ |D\tilde{h}|_{\partial\Omega}=1, \ \ \ D^2\tilde{h}\leq -\tilde{\theta}I,$$
where $\tilde{\theta}$ is some positive constant.

Let \begin{align*}
    \tilde{\beta}=(\tilde{\beta}^{1}, \cdots, \tilde{\beta}^{n}),\quad \tilde{\beta}^{k}:=\tilde{h}_{p_{k}}(D\tilde{u}),
\end{align*}
using the representation as the works of \cite{HRY} and \cite{OK},
we also define
$$\tilde{\omega}=\langle\tilde{\beta}, \tilde{\nu}\rangle+\tilde{\tau} \tilde{h}(D\tilde{u}),$$
in which
$$\langle\tilde{\beta}, \tilde{\nu}\rangle=\langle\beta, \nu\rangle,$$
and $\tilde{\tau}$ is a  positive constant to be determined.

Denote $y_{0}=Du(x_0)$, then $$\tilde{\omega}(y_{0})=\omega(x_{0})=\min_{\partial\tilde{\Omega}} \tilde{\omega}.$$

Using the same methods, under the assumption of
\begin{equation}\label{e3.9}
\tilde{\omega}_{n}(y_{0})\geq -C_{5},
\end{equation}
we obtain the positive lower bounds of $\tilde{h}_{p_{k}}\tilde{h}_{p_{l}}\tilde{u}_{kl}$ or
$$h_{p_{k}}(Du)\nu_{k}=\tilde{h}_{p_{k}}(D\tilde{u})\tilde{\nu}_{k}=\tilde{h}_{p_{n}}\geq
\sqrt{\frac{1}{2}+\frac{C^{4}_{6}}{4}}-\frac{C^{2}_{6}}{2}.$$

On the other hand, it is easy to check that
$$\tilde{h}_{p_{k}}\tilde{h}_{p_{l}}\tilde{u}_{kl}=\nu_{i}\nu_{j}u^{ij}.$$

Then by the positive lower bounds of $h_{p_{k}}h_{p_{l}}u_{kl}$ and $\tilde{h}_{p_{k}}\tilde{h}_{p_{l}}\tilde{u}_{kl}$, the desired conclusion can be obtained by
\begin{align*}
\langle \beta,\nu\rangle=\sqrt{h_{p_k}h_{p_l}u_{kl}u^{ij}\nu_i\nu_j}.
\end{align*}

For details of the proof of the above formula, the readers can refer to \cite{Ju2} or \cite{J2}.
It remains to prove the key estimates \eqref{e3.1.1} and \eqref{e3.9}.

At first we give the proof of \eqref{e3.1.1}. By (\ref{e3.3.4}), Corollary \ref{c3.3} and Lemma \ref{l3.30}, we have
\begin{equation*}\label{e3.10}
 \begin{aligned}
\mathcal{L}\omega=&G_{ij}u_{il}u_{jm}(h_{p_{k}p_{l}p_{m}}\nu_{k}+\tau h_{p_{l}p_{m}})\\
&+2G_{ij}h_{p_{k}p_{l}}u_{li}\nu_{kj}+G_{ij}h_{p_{k}}\nu_{kij}+G_{i}h_{p_{k}}\nu_{ki}\\
\leq& (h_{p_{k}p_{l}p_{m}}\nu_{k}+\tau h_{p_{l}p_{m}}+\delta_{lm})G_{ij}u_{il}u_{jm}+C_{7}\mathcal{T}_{G}+C_{8},
 \end{aligned}
\end{equation*}
where $\mathcal{L}=G_{ij}\partial_{ij}+G_{i}\partial_{i}$ and
$$2G_{ij}h_{p_{k}p_{l}}u_{li}\nu_{kj}\leq  G_{ij}u_{im}u_{mj}+C_{7}\mathcal{T}_{G}.$$

Since $D^{2}h\leq-\theta I$, we may choose $\tau$ large enough depending on the known data such that
$$(h_{p_{k}p_{l}p_{m}}\nu_{k}+\tau h_{p_{l}p_{m}}+\delta_{lm})<0.$$

Consequently, we deduce that
\begin{equation}\label{e3.11}
\mathcal{L}\omega\leq C_{9}\mathcal{T}_{G} \    \ in \   \ \Omega,
\end{equation}
by the convexity of $u$.

Using the method of construction in \cite{RS}, we denote a neighborhood of $x_{0}$ in $\Omega$ by
$$\Omega_{r}:=\Omega\cap B_{r}(x_{0}),$$
where $r$ is a positive constant such that $\nu$ is well defined in $\Omega_{r}$. In order to obtain the desired results, it suffices to consider the auxiliary function
$$\Phi(x)=\omega(x)-\omega(x_{0})+\sigma\ln(1+k\tilde{h}(x))+A|x-x_{0}|^{2},$$
where $\sigma$, $k$ and $A$ are positive constants to be determined.
By noting that $\tilde{h}$ is the defining function of $\Omega$ and $G_{i}$ is bounded, we show that
\begin{equation}\label{e3.12}
 \begin{aligned}
  \mathcal{L}(\ln(1+k\tilde{h}))&=G_{ij}\left(\frac{k\tilde{h}_{ij}}{1+k\tilde{h}}
  -\frac{k\tilde{h}_{i}}{1+k\tilde{h}}\frac{k\tilde{h}_{j}}{1+k\tilde{h}}\right)
  +G_{i}\frac{k\tilde{h}_{i}}{1+k\tilde{h}}\\
  &\triangleq G_{ij}\frac{k\tilde{h}_{ij}}{1+k\tilde{h}}-G_{ij}\eta_{i}\eta_{j}+G_{i}\eta_{i}\\
  &\leq\left(-\frac{k\tilde{\theta}}{1+k\tilde{h}}+C_{10}-C_{11}|\eta-C_{12}I|^{2}\right)\mathcal{T}_{G}\\
  &\leq\left(-\frac{k\tilde{\theta}}{1+k\tilde{h}}+C_{10}\right)\mathcal{T}_{G},
  \end{aligned}
\end{equation}
where $\eta=\left(\frac{k\tilde{h}_{1}}{1+k\tilde{h}}, \frac{k\tilde{h}_{2}}{1+k\tilde{h}},\cdots, \frac{k\bar{h}_{n}}{1+k\tilde{h}}\right)$.

By taking $r$ to be small enough, we have
\begin{equation}\label{e3.13}
 \begin{aligned}
0\leq \tilde{h}(x)&=\tilde{h}(x)-\tilde{h}(x_{0})\\
&\leq \sup_{\Omega_{r}}|D\tilde{h}||x-x_{0}|\\
&\leq r\sup_{\Omega}|D\tilde{h}|\leq \frac{\tilde{\theta}}{3C_{10}}.
 \end{aligned}
\end{equation}

By choosing $k=\frac{7C_{10}}{\tilde{\theta}}$ and applying (\ref{e3.13}) to (\ref{e3.12}), we obtain
\begin{equation}\label{e3.14}
  \mathcal{L}(\ln(1+k\tilde{h}))
  \leq -C_{10}\mathcal{T}_{G}.
\end{equation}

Combining (\ref{e3.11}) with (\ref{e3.14}), a direct computation yields
\begin{equation*}
\mathcal{L}(\Phi(x))\leq (C_{9}-\sigma C_{10}+2A)\mathcal{T}_{G}.
\end{equation*}
It is clear that $\Phi(x)\geq0$ on $\partial\Omega$.
Because $\omega$ is bounded, then it follows that we can choose $A$ large enough depending on the known data such that on $\Omega\cap\partial B_{r}(x_{0})$,
\begin{equation*}
 \begin{aligned}
  \Phi(x)&=\omega(x)-\omega(x_0)+\sigma\ln(1+kh^{\ast})+Ar^{2}\\
  &\geq\omega(x)-\omega x_0+Ar^{2}\geq 0.
  \end{aligned}
\end{equation*}
The above argument implies that
\begin{equation}\label{e3.15aaa}
  \Phi(x)\geq 0, \,\, \quad \text{for} \quad x \in \partial\Omega_{r}.
\end{equation}

By taking
$$\sigma=\frac{C_{9}+2A}{C_{10}},$$

\begin{equation}\label{e3.15}
   \mathcal{L}\Phi\leq 0,\,\,  x\in\Omega_{r}.
\end{equation}

According to the maximum principle, we get that
\begin{equation}\label{e3.15aaaa}
\Phi(x)|_{\Omega_{r}}\geq 0.
\end{equation}

By using the maximum principle in (\ref{e3.15}),
it follows from (\ref{e3.15aaa}) and (\ref{e3.15aaaa}) that
$$\Phi|_{\Omega_{r}}\geq \min_{\partial\Omega_{r}}\Phi\geq 0.$$

By the above inequality and $\Phi(x_0)=0$, we have $\partial_n\Phi(x_0)\geq 0$, which gives the desired estimate \eqref{e3.1.1}.

Finally, we turn to the proof of (\ref{e3.9}). The proof of (\ref{e3.9}) is similar to that of (\ref{e3.1.1}).
Define the elliptic operator
$$\mathcal{\tilde{L}}=\tilde{G}_{ij}\partial_{ij}.$$

By calculation, we arrive at
\begin{equation*}
 \begin{aligned}
\mathcal{\tilde{L}}\tilde{\omega}=&\tilde{G}_{ij}\tilde{u}_{li}\tilde{u}_{mj}(\tilde{h}_{q_{k}q_{l}q_{m}}\tilde{\nu}_{k}+\tilde{\tau} \tilde{h}_{q_{l}q_{m}})+2\tilde{G}_{ij}\tilde{h}_{q_{k}q_{l}}\tilde{u}_{li}\tilde{\nu}_{kj}\\
&-\tilde{G}_{y_{k}}(\tilde{h}_{q_{k}q_{m}}\tilde{\nu}_{m}+\tilde{\tau} \tilde{h}_{q_{k}})+\tilde{G}_{ij}\tilde{h}_{q_{k}}\tilde{\nu}_{kij}\\
\leq&(\tilde{h}_{q_{k}q_{l}q_{m}}\tilde{\nu}_{k}+\tilde{\tau} \tilde{h}_{q_{l}q_{m}}+\delta_{lm})\tilde{G}_{ij}\tilde{u}_{il}\tilde{u}_{jm}+C_{11}\mathcal{T}_{\tilde{G}}+C_{12}(1+\tilde{\tau}),
 \end{aligned}
\end{equation*}
where
$$2\tilde{G}_{ij}\tilde{h}_{q_{k}q_{l}}\tilde{u}_{li}\tilde{\nu}_{kj}\leq \delta_{lm}\tilde{G}_{ij}\tilde{u}_{il}\tilde{u}_{jm}
+C_{11}\mathcal{T}_{\tilde{G}},$$
by Lemma 3.6 in \cite{HR1}. Since $D^{2}\tilde{h}\leq-\tilde{\theta}I$, we only need to choose $\tilde{\tau}$ sufficiently large depending on the known data such that
$$\tilde{h}_{q_{k}q_{l}q_{m}}\tilde{\nu}_{k}+\tilde{\tau} \tilde{h}_{q_{l}q_{m}}+\delta_{lm}<0.$$
Therefore,
\begin{equation}\label{e3.16}
\mathcal{\tilde{L}}\tilde{\omega}\leq C_{13}\mathcal{T}_{\tilde{G}},
\end{equation}
by the convexity of $\tilde{u}$.

Denote a neighborhood of $y_{0}$ in $\tilde{\Omega}$ by
$$\tilde{\Omega}_{\rho}:=\tilde{\Omega}\cap B_{\rho}(y_{0}),$$
where $\rho$ is a positive constant such that $\tilde{\nu}$ is well defined in $\tilde{\Omega}_{\rho}$. In order to obtain the desired results,
we  consider the  auxiliary function
$$\Psi(y)=\tilde{\omega}(y)-\tilde{\omega}(y_{0})+\tilde{k}h(y)+\tilde{A}|y-y_{0}|^{2},$$
where $\tilde{k}$ and $\tilde{A}$ are positive constants to be determined. It is easy to check that $\Psi(y)\geq0$ on $\partial\tilde{\Omega}$. Note that $\tilde{\omega}$ is bounded, it follows that we can choose $\tilde{A}$ large enough depending on the known data, such that on $\tilde{\Omega}\cap\partial B_{\rho}(y_{0})$,
\begin{equation*}
  \Psi (y)=\tilde{\omega}(y)-\tilde{\omega}(y_{0})+\tilde{k}h(y)+\tilde{A}\rho^{2}\geq \tilde{\omega}y-\tilde{\omega}(y_{0})
  +\tilde{A}\rho^{2}\geq 0.
\end{equation*}

It follows from (\ref{e3.16}) and $D^{2}h\leq-\theta I$ that
$$\mathcal{\tilde{L}}\Psi\leq(C_{13}-\tilde{k}\theta+2\tilde{A})\mathcal{T}_{\tilde{G}},$$
Let $$\tilde{k}=\frac{2\tilde{A}+C_{13}}{\theta},$$
we consequently have
\begin{align*}
\begin{cases}
   \mathcal{\tilde{L}}\Psi\leq 0,\qquad  y\in\tilde{\Omega}_{\rho},\\
\quad\; \Psi\geq 0,\quad  y\in\partial\tilde{\Omega}_{\rho}.
\end{cases}
\end{align*}

The rest of the proof of (\ref{e3.9}) is the same as (\ref{e3.1.1}). Thus the proof of (\ref{e3.3.12}) is completed.
\end{proof}

\section{The global $C^2$ estimate}
We now proceed to carry out the $C^2$ estimate. The strategy is to construct suitable barrier function and use elliptic
maximum principle to reduce the $ C^2$ global estimates
of $u$ and $\tilde{u}$ to the boundary.
\begin{lemma}\label{lem4.1}
 Assume that $\Omega$, $\tilde{\Omega}$ are bounded, uniformly convex domains with smooth boundary in $\mathbb{R}^{n}$, $\tilde{\Omega}\subset\subset B_{1}(0)$. $0<\alpha<1$, $u\in C^{2+\alpha}(\bar{\Omega})$ is uniformly convex solution to \eqref{e1.1.3} and \eqref{e1.1.4}, then there exists  positive constants $C_{14}$,$C_{15}$,$C_{16}$ depending only on $u_0$, $\Omega$, $\tilde{\Omega}$, such that
\begin{equation}\label{eq3.15}
D^2 u(x) \leq C_{14} I_n,\ \ x\in\bar\Omega
\end{equation}
and
\begin{align}
\label{e4.4.2}
    C_{15}\leq u_{11}+u_{22}+\cdots+u_{nn}\leq C_{16},
\end{align}
where $I_n$ is the $n\times n$ identity matrix.
\end{lemma}
\begin{proof}
By the proof of \eqref{e2.2.8} in Lemma \ref{L2.1}, we can show that
$F=\sum_{i=1}^n\kappa_{i}$ is bounded.
From the formula in (\ref{e2.2.13}), we get that
$$\sum_{i=1}^n\kappa_{i}=\sum_{i=1}^n\frac{1}{v}b^{ik}D_{kl}ub^{li}.$$
Then, by using \begin{equation*}
b^{ij}=\delta_{ij}+\frac{D_{i}uD_{j}u}{v(1+v)},\;
|Du|<1,\;
v=\sqrt{1-|Du|^2}
\end{equation*}
and the second boundary condition,   we obtain $D^2 u(x) $ is bounded.

In addition, by Lemma \ref{L2.1}, we know
\begin{align*}
    \Lambda_1\leq\kappa_1+\cdots +\kappa_n\leq \Lambda_2
\end{align*}
and then we obtain \eqref{e4.4.2}.
\end{proof}
In the following,  we derive the positive lower bound of $D^{2}u$.
To obtain  the positive lower bound of $D^{2}u$ on $\partial\Omega$, we consider the Legendre transformation of $u$.
As before, we see that this can in fact be written
\begin{equation*}
\frac{\partial \tilde{u}}{\partial y_{i}}=x_{i},\,\,\frac{\partial ^{2}\tilde{u}}{\partial y_{i}\partial y_{j}}=u^{ij}(x),
\end{equation*}
where $[u^{ij}]=[D^{2}u]^{-1}$.

By using \eqref{transfoemation G} and noting that $\tilde{h}$ is the defining function of $\Omega$, then we know $\tilde{u}$ satisfies
\begin{align}\label{e2.17}
 \begin{cases}
\tilde{G}(y,D^{2}\tilde{u})=-c,\quad\tilde{x}\in \tilde{\Omega}, \\
\qquad\tilde{h}(D\tilde{u})=0,\quad \tilde{x}\in\partial\tilde{\Omega},\\
\end{cases}
\end{align}
where $\tilde{G}(y,D^{2}\tilde{u})=-G(y,D^{2}\tilde{u}^{-1})$.

The following Lemma is to reduce the global $C^2-$ estimates of $\tilde{u}$ to the boundary.

\begin{lemma}\label{lem4.3}
     Let $\tilde{G}=\tilde{G}(y,D^{2}\tilde{u})=-c$.If $\tilde{u}$ is a smooth uniformly convex solution of \eqref{e2.17} and there hold \eqref{e2.2.8}-\eqref{e2.2.12}, then there exists a positive constant $C_{17}$ depending only on $n, \Omega, \tilde{\Omega}$ and $\operatorname{diam}(\Omega)$, such that
$$
\sup _{\Omega}\left|D^2 \tilde{u}\right| \leq \max _{\partial \Omega}\left|D^2 \tilde{u}\right|+C_{17}.
$$
\end{lemma}
\begin{proof}
Denote
\begin{align*}
    s_{ij}=\frac{1}{\sqrt{1-|y|^2}}\left(\delta_{ij}+\frac{y_{i}y_{j}}{1-|y|^{2}}\right),\quad[\tilde{u}_{ij}]=[u_{ij}]^{-1}.
\end{align*}
And then
\begin{align*}
    \frac{1}{\sqrt{1-|y|^2}} I_n \leq [s_{ij}]\leq \frac{n}{\left(1-|y|^2\right)^{\frac{3}{2}}}I_n,
\end{align*}
where $I_n$ is the $n \times n$ identity matrix. By calculating, we show that
 \begin{align*}
   G(y,D^2u)&=\mathrm{div} \left(\frac{Du}{\sqrt{1-|Du|^2}}\right)\\
   &=\frac{1}{\sqrt{1-|Du|^2}}\left(\delta_{ij}+\frac{y_{i}y_{j}}{1-|y|^{2}}\right)u_{ij},
\end{align*}
and then
$$\tilde{G}(y, D^{2}\tilde{u})=-s_{ij}u_{ij}.$$
By the rotation of the coordinate system for any fixed point $y_0\in \tilde{\Omega}$, such that  we can get at $y_0=Du(x_0)$
$$\left[D^2{\tilde{u}}\right]_{|y_0}=\mathrm{diag}(\tilde{u}_{11},\tilde{u}_{22},\cdots\tilde{u}_{nn}).$$
\begin{align*}
\left[D^2u\right]_{|y_0}=\mathrm{diag}(u_{11},u_{22},\cdots,u_{nn}).
\end{align*}

For any $1\leq k\leq n$, the second derivation of $y_k$ with respect to both sides of $s_{ij}u_{ij}=-c$, we can obtain
\begin{align*}
    s_{ij,kk}u_{ij}+s_{ij,k}\frac{\partial u_{ij}}{\partial y_k}+s_{ij,k}\frac{\partial u_{ij}}{\partial y_k}+s_{ij}\frac{\partial^2u_{ij}}{\partial y_k\partial y_{k}}=0,
\end{align*}
\begin{align}\label{e4.4.4}
     2s_{ij,k}\frac{\partial u_{ij}}{\partial y_k}+s_{ij}\frac{\partial^2u_{ij}}{\partial y_k\partial y_{k}}=-  s_{ij,kk}u_{ij},
\end{align}
Because \begin{align*}u_{ij}\tilde{u}_{js}=\delta_{is},
\end{align*}
and then
\begin{align*}
    \frac{\partial u_{ij}}{\partial y_k}=u_{ii}u_{jj}\tilde{u}_{ijk},\quad
    \frac{\partial^2 u_{ij}}{\partial y_k\partial y_k}=-u_{ii}u_{jj}\tilde{u}_{ijkk}+2u_{ii}\tilde{u}_{irk}u_{rr}\tilde{u}_{rjk}u_{jj
    }.
\end{align*}
 where are the results of diagonalisation.

Let $\lambda_{ij,k}=u_{ii}u_{jj}\tilde{u}_{ijk}$, \eqref{e4.4.4} can be written as
\begin{align*}
s_{ij}u_{ii}u_{jj}\tilde{u}_{ijkk}&=2s_{ij}u_{ii}u_{jj}\tilde{u}_{irk}u_{rr}\tilde{u}_{rjk}-2s_{ij,k}u_{ii}u_{jj}\tilde{u}_{ijk}+s_{ij,kk}u_{ij}\\
&=2s_{ij}\lambda_{ir,k}u_{jj}\tilde{u}_{rjk}-2s_{ij,k}\lambda_{ij,k}+s_{ij,kk}u_{ij}\\
&=2s_{ij}\lambda_{ir,k}\lambda_{jr,k}u^{-1}_{rr}-2s_{ij,k}\lambda_{ij,k}+s_{ij,kk}u_{ij}\\
&\geq C_{18}(\lambda_{ir,k}-C_{19})^2-C_{20}\\
&\geq -C_{20}.
\end{align*}
where we have used Lemma \ref{lem4.1}, $C_{18},C_{19},C_{20}$ are positive constants.

Without loss of generality,  we may assume that $\Omega$ lies in cube $[0, d]^n$. Let
\begin{equation}\label{e4.4.3}
    \tilde{\mathcal{L}}=\tilde{a}_{pq}\partial^2_{pq}.
\end{equation}
where $\tilde{a}_{pq}=s_{ij}u_{ip}u_{jq}$.
Because $[D^2u]$ is diagonal, and we can obtain
\begin{align*}
\sum_{i=1}^n\tilde{a}_{ii}&=\sum_{i,j=1}^n s_{ij}u_{ii}u_{jj}\\
    &\geq \frac{1}{\left(1-|y|^2\right)^{\frac{1}{2}}}(u^2_{11}+u^2_{22}+\cdots+u^2_{nn})\\
    &\geq\frac{1}{\left(1-|y|^2\right)^{\frac{1}{2}}}\frac{(u_{11}+u_{22}+\cdots+u_{nn})^2}{n}
\end{align*}
and
\begin{align*}
\sum_{i=1}^n\tilde{a}_{ii}&=\sum_{i,j=1}^n s_{ij}u_{ii}u_{jj}\\
    &\leq \frac{n}{\left(1-|y|^2\right)^{\frac{3}{2}}}(u^2_{11}+u^2_{22}+\cdots+u^2_{nn})\\
    &\leq\frac{n}{\left(1-|y|^2\right)^{\frac{3}{2}}}(u_{11}+u_{22}+\cdots+u_{nn})^2.
\end{align*}

So there exist two positive constants $C_{21},C_{22}$, such that
\begin{align*}
    C_{21}\leq \sum_{i=1}^n\tilde{a}_{ii}&\leq C_{22}.
\end{align*}
Let
\begin{align*}
    w=\max_{\partial\Omega}(\tilde{u}_{kk})+C_{23}(ne^d-(e^{x_1}+\cdots+e^{x_n})),
\end{align*}
obviously,
\begin{align*}
    w-\tilde{u}_{kk}\geq 0,\qquad y\in\partial\Omega.
\end{align*}
\begin{align*}
\tilde{\mathcal{L}}w=\tilde{a}_{pq}\partial^2_{pq}w=-C_{23}(\tilde{a}_{11}e^{x_1}+\cdots+\tilde{a}_{nn}e^{x_n}),
\end{align*}
and then we can obtain
\begin{align}
     \tilde{\mathcal{L}}(w-\tilde{u}_{kk})&\leq -C_{23}(\tilde{a}_{11}e^{x_1}+\cdots+\tilde{a}_{nn}e^{x_n})-\tilde{a}_{ij}\tilde{u}_{ijkk}\\
     &\leq -C_{23}(\tilde{a}_{11}+\cdots+\tilde{a}_{nn})e^{-d}+C_{20}\\
     &\leq -C_{23}C_{21}e^{-d}+C_{20},
\end{align}
when $C_{23}=\frac{C_{20}}{C_{21}}e^d$, so we get
\begin{align*}
     \tilde{\mathcal{L}}(w-\tilde{u}_{kk})\leq 0,\quad y\in\tilde{\Omega}.
\end{align*}
Then by the maximum principle, for any $y\in \tilde{\Omega}$. We obtain
\begin{align*}
    \tilde{u}_{kk}\leq w\leq \max_{\partial\Omega}|D^2\tilde{u}|+C_{17}.
\end{align*}
 Here $C_{17}=C_{23}ne^d$.
This completes the proof.
\end{proof}
\begin{lemma}\label{lem4.2}
Let $\tilde{a}_{pq}=s_{ij}u_{ip}u_{jq}$, $\mathcal{\tilde{L}}=\tilde{a}_{pq}\partial^2_{pq}$, $\tilde{G}_{ij}=\frac{\partial\tilde{G}}{\partial\tilde{u}_{ij}}.$   If $\tilde{u}$ is a strictly convex solution of \eqref{e2.17},
then there exists a positive constant $C_{24}$ depending only on $\Omega$, $\tilde{\Omega}$, such that
\begin{equation}\label{e3.42a}
\mathcal{\tilde{L}}\tilde{u}_{kk}\geq -C_{24}\sum_{i=1}^n\tilde{G}_{ii}.
\end{equation}
\end{lemma}
\begin{proof}
By the proof of Lemma \ref{lem4.3}, we know that
\begin{align*}
    \mathcal{\tilde{L}}\tilde{u}_{kk}\geq -C_{20},
\end{align*}
and by Corollary \ref{c3.4}, we can conclude that
\begin{align*}
\mathcal{\tilde{L}}\tilde{u}_{kk}\geq -C_{24}\sum_{i=1}^n\tilde{G}_{ii},
\end{align*}
here $C_{24}=\frac{C_{20}}{\Lambda_7}$.
\end{proof}

Recalling that $\tilde{\beta}=(\tilde{\beta}^{1}, \cdots, \tilde{\beta}^{n})$ with $\tilde{\beta}^{k}:=\tilde{h}_{p_{k}}(D\tilde{u})$ and $\tilde{\nu}=(\tilde{\nu}_{1}, \tilde{\nu}_{2},\cdots,\tilde{\nu}_{n})$ is the unit inward normal vector of $\partial\tilde{\Omega}$.
In the following we give the arguments as in \cite{JU2}, the readers can see there for more details.
For any tangential direction $\tilde \varsigma$, we have
\begin{equation}\label{e3.35}
   \tilde{u}_{\tilde\beta \tilde\varsigma}=\tilde h_{p_k}(D\tilde u)\tilde u_{k\tilde\varsigma}=0.
\end{equation}

Then the second order derivative of $\tilde u$ on the boundary is also controlled by $u_{\tilde\beta \tilde\varsigma}$, $u_{\tilde\beta \tilde\beta}$ and $u_{\tilde\varsigma\tilde\varsigma}$. At $\tilde x\in \partial\tilde\Omega$, any unit vector $\tilde\xi$ can be written in terms of a tangential component $\tilde\varsigma(\tilde\xi)$ and a component in the direction $\tilde\beta$ by
$$\tilde\xi=\tilde\varsigma(\tilde\xi)+\frac{\langle \tilde\nu,\tilde\xi\rangle}{\langle\tilde\beta,\tilde\nu\rangle}\tilde\beta,$$
where
$$\tilde\varsigma(\tilde\xi):=\tilde\xi-\langle \tilde\nu,\tilde\xi\rangle \tilde\nu-\frac{\langle \tilde\nu,\tilde\xi\rangle}{\langle\tilde\beta,\tilde\nu\rangle}\tilde\beta^T$$
and
$$\tilde\beta^T:=\tilde\beta-\langle \tilde\beta,\tilde\nu\rangle \tilde\nu.$$

We observe that $\langle\tilde\beta,\tilde\nu\rangle=\langle\beta,\nu\rangle$.
By the uniformly obliqueness estimate (\ref{e3.3.12}), we have
\begin{equation}\label{e3.36}
\begin{aligned}
|\tilde{\varsigma}(\tilde{\xi})|^{2}&=1-\left(1-\frac{|\tilde{\beta}^{T}|^{2}}{\langle\tilde{\beta},\tilde{\nu}\rangle^{2}}\right)
\langle\tilde{\nu},\tilde{\xi}\rangle^{2}
-2\langle\tilde{\nu},\tilde{\xi}\rangle\frac{\langle\tilde{\beta}^{T},\tilde{\xi}\rangle}{\langle\tilde{\beta},\tilde{\nu}\rangle}\\
&\leq 1+C_{25}\langle\tilde{\nu},\tilde{\xi}\rangle^{2}-2\langle\tilde{\nu},\tilde{\xi}\rangle\frac{\langle\tilde{\beta}^{T},\tilde{\xi}\rangle}{\langle\tilde{\beta},\tilde{\nu}\rangle}\\
&\leq C_{26}.
\end{aligned}
\end{equation}

Denote $\varsigma:=\frac{\varsigma(\xi)}{|\varsigma(\xi)|}$, then by (\ref{e3.3.12}), (\ref{e3.35}) and (\ref{e3.36}), we arrive at
\begin{equation}\label{e4.4.5}
\begin{aligned}
\tilde{u}_{\tilde{\xi}\tilde{\xi}}&=|\tilde{\varsigma}(\tilde{\xi})|^{2}
\tilde{u}_{\tilde{\varsigma}\tilde{\varsigma}}+2|\tilde{\varsigma}(\tilde{\xi})|\frac{\langle\tilde{\nu},\tilde{\xi}\rangle}{\langle\tilde{\beta},\tilde{\nu}\rangle}\tilde{u}_{\tilde{\beta}\tilde{\varsigma}}+
\frac{\langle\nu,\xi\rangle^{2}}{\langle\beta,\nu\rangle^{2}}
\tilde{u}_{\tilde{\beta}\tilde{\beta}}\\
&=|\tilde{\varsigma}(\tilde{\xi})|^{2}\tilde{u}_{\tilde{\varsigma}\tilde{\varsigma}}+\frac{\langle\tilde{\nu},\tilde{\xi}\rangle^{2}}{\langle\tilde{\beta},\tilde{\nu}\rangle^{2}}
\tilde{u}_{\tilde{\beta}\tilde{\beta}}\\
&\leq C_{27}(\tilde{u}_{\tilde{\varsigma}\tilde{\varsigma}}+\tilde{u}_{\tilde{\beta}\tilde{\beta}}).
\end{aligned}
\end{equation}

Therefore, we also only need to estimate $\tilde{u}_{\tilde\beta\tilde\beta}$ and $\tilde{u}_{\tilde\varsigma\tilde\varsigma}$ respectively.
\begin{lemma}\label{lem4.4}
If $\tilde{u}$ is a strictly convex solution of \eqref{e2.17}, then there exists a positive constant $C_{28}$ depending only on $\Omega$, $\tilde{\Omega}$, such that
\begin{equation}\label{e4 4.6}
   \max_{\partial\Omega}\tilde{u}_{\tilde{\beta}\tilde{\beta}} \leq C_{28}.
\end{equation}
\end{lemma}
\begin{proof}
Let $\tilde{x}_0\in\partial\tilde{\Omega}$, such that $\tilde{u}_{\tilde{\beta}\tilde{\beta}}(\tilde{x}_0)=\max_{\partial\Omega}\tilde{u}_{\tilde{\beta}\tilde{\beta}}$.
To estimate the upper bound of $\tilde{u}_{\tilde{\beta}\tilde{\beta}}$,
we consider the barrier function
$$\tilde\Psi:=-\tilde{h}(D\tilde{u})+C_0 h.$$

For any $y\in \partial\tilde{\Omega}$, $D\tilde{u}(y)\in \partial\Omega$, then $\tilde{h}(D\tilde{u})=0$. It is clear that $h=0$ on $\partial\tilde{\Omega}$.  First we have
$$\mathcal{\tilde{L}}(C_0 h)=C_0 \tilde{G}_{ij}h_{ij}\leq C_0\left(-\theta\sum_{i=1}^{n}\tilde{G}_{ii}\right).$$

Using the equations (\ref{e2.17}), a direct computation shows that
\begin{equation}\label{e3.39}
\begin{aligned}
\mathcal{\tilde{L}}\left(-\tilde{h}(D\tilde{u})\right)&=\tilde{G}_{ij}\left(-\tilde{h}_{\tilde{p}_{k}\tilde{p}_{l}}
\partial_{ki}\tilde{u}\partial_{lj}\tilde{u}\right)-\tilde{h}_{\tilde{p}_{k}}\tilde{G}_{y_{k}}\\
&\leq C_{29}\sum_{i=1}^{n}\tilde{G}_{ii},
\end{aligned}
\end{equation}
where we use the estimates (\ref{e3..3.3})-(\ref{e3..3.4}) in Corollary \ref{c3.4}.
Therefore, we obtain
$$\tilde {\mathcal{L}}\tilde\Psi(y)\leq \left(C_{29}-C_0\theta\right)\sum_{i=1}^n \tilde G_{ii}. $$
Let \begin{align*}
C_0=\frac{C_{29}}{\theta}.
\end{align*}
It is clear that $\tilde\Psi=0$ on $\partial\tilde\Omega$.
It follows from the above results that
\begin{equation}\label{e3.40}
\left\{ \begin{aligned}
   \mathcal{\tilde{L}}\tilde\Psi&\leq 0,\quad &&y\in\tilde\Omega,\\
   \tilde\Psi&\geq 0 ,\quad\  &&y\in\partial\tilde\Omega.
\end{aligned} \right.
\end{equation}
Applying the maximum principle, we get
$$\tilde\Psi(y)\geq 0,\quad\quad y\in \tilde{\Omega}.$$
Combining it with $\tilde\Psi(\tilde{x}_0)=0$, we obtain $\tilde\Psi_{\tilde{\beta}}(\tilde{x}_0)\geq 0$, which implies
$$\frac{\partial \tilde{h}}{\partial \tilde{\beta}}(D\tilde{u}(\tilde{x}_0))\leq C_0.$$
On the other hand, we see that at $\tilde{x}_0$,
$$\frac{\partial \tilde{h}}{\partial \tilde{\beta}}=\langle D\tilde{h}(D\tilde{u}),\tilde{\beta}\rangle=\frac{\partial \tilde{h}}{\partial p_k}\tilde{u}_{kl}\tilde{\beta}^l=\tilde{\beta}^k\tilde{u}_{kl}\tilde{\beta}^l=\tilde{u}_{\tilde{\beta}\tilde{\beta}}.$$
Therefore, letting $C_{28}= C_{0}$ we have
$$\tilde{u}_{\tilde{\beta}\tilde{\beta}}=\frac{\partial \tilde{h}}{\partial \tilde{\beta}}\leq C_{28}.$$
\end{proof}
\begin{lemma}\label{lem4.5}
If $\tilde{u}$ is a strictly convex solution of \eqref{e2.17}, then there exists a positive constant $C_{30}$ depending only on $u_0$, $\Omega$, $\tilde{\Omega}$, such that
\begin{equation*}\label{e4.43}
   \max_{\partial\tilde{\Omega}}\max_{|\tilde\varsigma|=1, \langle\tilde\varsigma,\tilde\nu\rangle=0} \tilde{u}_{\tilde\varsigma\tilde\varsigma} \leq C_{30}.
\end{equation*}
\end{lemma}
\begin{proof}
We assume that $\tilde{x}_{0}\in\partial\Omega$, $e_{n}$ is the unit inward normal vector of $\partial\tilde{\Omega}$ at $\tilde{x}_0$ and $e_{1}$ is the tangential vector of $\partial\tilde{\Omega}$ at $\tilde{x}_0$ respectively, such that
$$ \max_{\partial\tilde{\Omega}}\max_{|\tilde\varsigma|=1, \langle\tilde\varsigma,\tilde\nu\rangle=0} \tilde{u}_{\tilde\varsigma\tilde\varsigma}=\tilde{u}_{11}(\tilde{x}_0)=:\mathcal{ M}.$$
For any $y\in \partial\tilde{\Omega}$, it follows from the proof of (\ref{e3.36}) and (\ref{e4.4.5}) that
\begin{equation}\label{e4.45}
\begin{aligned}
\tilde u_{\tilde\xi\tilde\xi}&=|\tilde\varsigma(\tilde\xi)|^2\tilde u_{\tilde\varsigma\tilde\varsigma}+ \frac{\langle \tilde\nu,\tilde\xi\rangle^2}{\langle \tilde\beta,\tilde\nu\rangle^2}\tilde u_{\tilde\beta\tilde\beta}\\
          &\leq \left(1+C_{31}\langle \tilde\nu,\tilde\xi\rangle^2-2\langle \tilde\nu,\tilde\xi\rangle \frac{\langle \tilde\beta^T,\tilde\xi\rangle}{\langle \tilde\beta,\tilde\nu\rangle}\right) \mathcal{M}
                 + \frac{\langle \tilde\nu,\tilde\xi\rangle^2}{\langle \tilde\beta,\tilde\nu\rangle^2}\tilde u_{\tilde\beta\tilde\beta}.
\end{aligned}
\end{equation}

Let us skip therefore to the case $\mathcal{M}\geq 1$. Thus by (\ref{e3.3.12}), (\ref{e4 4.6}) and (\ref{e4.45}) we deduce that
\begin{equation*}\label{eq3.9a}
  \frac{\tilde u_{\tilde\xi\tilde\xi}}{\mathcal {M}}+2\langle \tilde\nu,\tilde\xi\rangle \frac{\langle \tilde\beta^T,\tilde\xi\rangle}{\langle \tilde\beta,\tilde\nu\rangle}
      \leq 1+C_{32}\langle \tilde\nu,\tilde\xi\rangle^2.
\end{equation*}
Let $\tilde\xi=e_1$, then
\begin{equation*}\label{eq3.10a}
  \frac{\tilde u_{11}}{\mathcal{ M}}+2\langle \tilde\nu,e_1\rangle \frac{\langle \tilde\beta^T,e_1\rangle}{\langle \tilde\beta,\tilde\nu\rangle}
             \leq 1+C_{32}\langle \tilde\nu,e_1\rangle^2.
\end{equation*}
We see that the function
\begin{equation*}\label{eq3.11a}
\tilde w:=A|y-\tilde x_0|^2-\frac{\tilde u_{11}}{\mathcal{ M}}-2\langle \tilde\nu,e_1\rangle \frac{\langle \tilde\beta^T,e_1\rangle}{\langle \tilde\beta,\tilde\nu\rangle}+C_{32}\langle \tilde\nu,e_1\rangle^2+1
\end{equation*}
satisfies
$$\tilde w|_{\partial\tilde\Omega}\geq 0,\quad  \tilde w(\tilde x_0)=0.$$
Denote a neighborhood of $\tilde{x}_0$ in $\tilde\Omega$ by
$$\tilde\Omega_{r}:=\tilde\Omega\cap B_{r}(\tilde{x}_0),$$
where $r$ is a positive constant such that $\tilde\nu$ is well defined in $\tilde\Omega_{r}$.
Let us consider
$$-2\langle \tilde\nu,e_1\rangle \frac{\langle \tilde\beta^T,e_1\rangle}{\langle \tilde\beta,\tilde\nu\rangle}+C_{32}\langle \tilde\nu,e_1\rangle^2+1$$
as a known function depending on $y$ and $D\tilde u$. Then by the proof of (\ref{e3.39}), we also obtain
\begin{equation*}
\left|\mathcal{\tilde L}\left(-2\langle \tilde\nu,e_1\rangle \frac{\langle \tilde\beta^T,e_1\rangle}{\langle \tilde\beta,\tilde\nu\rangle}+C_{32}\langle \tilde\nu,e_1\rangle^2+1\right)\right|\leq C_{33}\sum_{i=1}^n \tilde{G}_{ii}.
\end{equation*}
It follows from (\ref{e3.42a}) in Lemma $\ref{lem4.2}$ that
\begin{equation*}
\mathcal{\tilde{L}}\tilde{u}_{11}\geq-C_{24}\sum_{i=1}^{n}\tilde{G}_{ii}.
\end{equation*}
We set
$$\tilde\Upsilon:=\tilde w+C_0 {h}.$$
Furthermore, by(\ref{e3.3.12}), (\ref{e3.3.14}), (\ref{e4.4.5}) and (\ref{e4 4.6}),  we can choose the constant $A$ large enough such that
$$\tilde w|_{\tilde\Omega \cap \partial B_{r}(\tilde x_0)} \geq 0. $$
As in the proof of (\ref{e3.40}),
\begin{align*}
    \tilde{\mathcal{L}}(\tilde{w}+C_0h)&=\tilde{\mathcal{L}}\left(A|y-\tilde x_0|^2-\frac{\tilde u_{11}}{\mathcal{ M}}-2\langle \tilde\nu,e_1\rangle \frac{\langle \tilde\beta^T,e_1\rangle}{\langle \tilde\beta,\tilde\nu\rangle}+C_{32}\langle \tilde\nu,e_1\rangle^2+1+C_0h\right)\\
   & \leq 2A\sum_{i=1}^{n}\tilde{G}_{ii}+\frac{C_{24}}{\mathcal{M}}\sum_{i=1}^{n}\tilde{G}_{ii}+C_{33}\sum_{i=1}^{n}\tilde{G}_{ii}+-C_0\theta\sum_{i=1}^{n}\tilde{G}_{ii}\\
   &\leq \left(2A+C_{24}+C_{33}+-C_0\theta\right)\sum_{i=1}^{n}\tilde{G}_{ii}\\
   &\leq \left(2A+C_{24}+C_{33}+-C_0\theta\right)\Lambda_7.
\end{align*}
 where $C_0=\frac{2A+C_{24}+C_{33}}{\theta}$, we get that
 \begin{align*}
     \mathcal{\tilde{L}}\tilde\Upsilon\leq 0,\quad y\in \tilde\Omega_{r}.
 \end{align*}

A standard barrier argument makes conclusion of
$$\tilde\Upsilon_{\tilde\beta}(\tilde x_0)\geq0.$$
Therefore,
\begin{equation}\label{eq3.12a}
 \tilde u_{11\tilde\beta}(\tilde x_0)\leq C_{34}\mathcal{ M}.
\end{equation}

On the other hand, differentiating $\tilde h(D\tilde u)$ twice in the direction $e_1$ at $\tilde x_0$, we have
$$\tilde h_{p_k}\tilde u_{k11}+\tilde h_{p_kp_l}\tilde u_{k1}\tilde u_{l1}=0.$$
The concavity of $\tilde h$ yields that
$$\tilde h_{p_k}\tilde u_{k11}=-\tilde h_{p_kp_l}\tilde u_{k1}\tilde u_{l1}\geq \tilde\theta \mathcal{M}^2.$$
Combining it with $\tilde h_{p_k}\tilde u_{k11}=\tilde u_{11\tilde\beta}$ and using (\ref{eq3.12a}), we obtain
$$\tilde\theta  \mathcal{M}^2\leq C_{34}\mathcal{M}.$$
Then we get the upper bound of $\mathcal{M}=\tilde u_{11}(\tilde x_0)$ and thus the desired result follows.
\end{proof}
By Lemma \ref{lem4.4}, Lemma \ref{lem4.5} and (\ref{e4.4.5}), we obtain the $C^2$ a-priori estimate of $\tilde{u}$ on the boundary.
\begin{lemma}\label{lem4.6}
If $\tilde{u}$ is a strictly convex solution of \eqref{e2.17}, then there exists a positive constant $C_{35}$ depending only on $u_0$, $\Omega$, $\tilde{\Omega}$, such that
\begin{equation*}\label{eq3.13a}
\max_{\partial\tilde\Omega}|D^2\tilde u| \leq C_{35}.
\end{equation*}
\end{lemma}

By Lemma \ref{lem4.3} and Lemma \ref{lem4.6}, we can see that
\begin{lemma}\label{lem4.7}
If $\tilde{u}$ is a strictly convex solution of \eqref{e2.17}, then there exists a positive constant $C_{36}$ depending only on $u_0$, $\Omega$, $\tilde{\Omega}$, such that
\begin{equation*}\label{eq3.14a}
\max_{\bar{\tilde\Omega}}|D^2\tilde u| \leq C_{36}.
\end{equation*}
\end{lemma}

By Lemma \ref{lem4.1} and Lemma \ref{lem4.7}, we conclude that
\begin{lemma}\label{lem3.6}
Assume that $\Omega$, $\tilde{\Omega}$ are bounded, uniformly convex domains with smooth boundary in $\mathbb{R}^{n}$ and $\tilde{\Omega}\subset\subset B_{1}(0)$. If $u$ is a strictly convex solution to \eqref{e1.1.3}-\eqref{e1.1.4}, then there exists a positive constant $C_{37}$ depending only on $\Omega$, $\tilde{\Omega}$, such that
\begin{equation*}\label{eq4.4.8}
\frac{1}{C_{37}}I_n\leq D^2 u(x) \leq C_{37} I_n,\; x\in\bar\Omega,
\end{equation*}
where $I_n$ is the $n\times n$ identity matrix.
\end{lemma}
\vspace{3mm}
\section{The invertibility of linearized operators}
In this section, we prove that all solutions of (\ref{e1.1.3}) and (\ref{e1.1.4}) are non-degenerate. We fix two strictly convex domains $\Omega,\tilde{\Omega}$ in $\mathbb{R}^n$. We also fix two points $\overline{p}\in \partial \Omega$ and $\overline{q}\in \partial \tilde{\Omega}$. Suppose that $f:\Omega \to\tilde{\Omega} $ is onto and one-to-one. We claim that linearized operator at $f$ is invertible.

In order to prove this, we fix a real number $\alpha\in(0,1)$. We denote the Banach space by
\begin{align*}
    \mathcal{X}=\{u\in C^{2,\alpha}(\overline{\Omega}):\int_{\Omega}u=0\}
\end{align*}
and
\begin{align*}
    \mathcal{Y}=C^{\alpha}(\overline{\Omega})\times C^{1,\alpha}(\partial \Omega).
\end{align*}
We define a map $\mathcal{F}:\mathcal{X}\times\mathbb{R}\to\mathcal{Y}$ by
\begin{align*}
    \mathcal{F}(u,c)=(F(Du,D^2u)-c,(\tilde{h}\circ(\nabla u)|_{\partial \Omega}).
\end{align*}
We consider the following problem:
\begin{problem}\label{prob1}
Find a convex function $u: \Omega\to \mathbb{R}$ and a bounded constant $c$ such that $\nabla u$ is a diffeomorphism from $\Omega$ to $\tilde{\Omega}$ and $F(Du(x),D^2u(x))=c$ for all $x\in \Omega$.
\end{problem}
Therefore, if $(u,c)\in \mathcal{X}\times \mathbb{R}$ is a solution of Problem \ref{prob1}, and then $\mathcal{F}(u,c)=(0,0)$.\\
We next define a linearized operator $D\mathcal{F}:\mathcal{X}\times \mathbb{R}\to \mathcal{Y}$ by
\begin{align*}
    D\mathcal{F}(w,a)=(\mathcal{L}w-a,Nw).
\end{align*}
The operator $\mathcal{L}:C^{2,\alpha}(\overline{\Omega})\to C^{\alpha}(\overline{\Omega})$ is define by
\begin{align*}
    \mathcal{L}w(x)=G_{ij}(Du,D^2u)\partial_{ij}w-G_{p_i}(Du,D^2u)\partial_iw,
\end{align*}
for $x\in \Omega$. Moreover, the operator $N:C^{2,\alpha}(\overline\Omega)\to C^{1,\alpha}(\partial \Omega)$ is defined by
\begin{align*}
    Nw(x)=\langle \nabla w(x),\nabla \tilde{h}(\nabla u(x))\rangle
\end{align*}
for $x\in\partial \Omega$.
Obviously, $\mathcal{L}$ is an elliptic operator, the boundary condition is oblique.
\begin{proposition}\label{e4.4.1}
    The linearized operator $D\mathcal{F}:\mathcal{X}\times \mathbb{R}\to \mathcal{Y}$ is invertible.
\end{proposition}
\begin{proof}
    We claim that the operator $D\mathcal{F}$ is one-to-one. To prove this, suppose that $w$ is a real-valued function such that $\mathcal{L}w-a=0$ in $\Omega$ and $Nw(x)=0$. This implies that $\mathcal{L}w=a$ for all $x\in \Omega$. If the constant $a$ is positive, then $\mathcal{L}w\leq 0$, $w$ is strictly negative in the interior of $\Omega$ by the maximum principle. Hence, the Hopf boundary point lemma (cf. \cite{GT}, Lemma 3.4)
implies that the outer normal derivative of $w$  is strictly positive.
This contradicts the fact that $Nw(x)=0$.

Thus, we conclude that
$a\leq 0$. An analogous argument shows that $a\geq 0$. Consequently, we
must have $a=0$. Using the maximum principle, we deduce that $w=0$.
Thus, the operator $D\mathcal{F}:\mathcal{X}\times \mathbb{R}\to \mathcal{Y}$ is one-to-one.

A similar argument shows that $D\mathcal{F}:\mathcal{X}\times \mathbb{R}\to \mathcal{Y}$ is onto. This
completes the proof.
\end{proof}
\vspace{3mm}
\section{Existence and uniqueness of solutions}
In this section, we prove the existence of a solution to (\ref{e1.1.3}) and (\ref{e1.1.4}) by the continuity method
and some tricks which we learn from Brendle and Warren's work \cite{SM}.
The proof of uniqueness up to a constant, we refer the readers to Lemma 5.1 in Huang-Li \cite{RS}.

Let $\Omega$ and $\tilde{\Omega}$ be uniformly convex domains in $\mathbb{R}^n$ with smooth boundary.
To prove the existence of the solution, we first show that Theorem \ref{t1.1} holds
 when $\Omega$ and $\tilde{\Omega}$ are two balls in $ R^n $.
 Then we deform the balls to the given strictly convex domains and prove the existence of general
 case of Theorem \ref{t1.1}.
\begin{lemma}\label{lem6.1}
 Let $B(0,R_0)$ and $B(0,t_0)$$(t_0<1)$ are bounded, uniformly convex domains with smooth boundary in $\mathbb{R}^{n}$,
then the following problem
    \begin{equation}
\begin{cases}\label{e6.6.1}
    \mathrm{div}\left(\frac{Du}{\sqrt{1-|Du|^2}}\right)=c,\quad x\in B(0,R_0),\\
Du(B(0,R_0))=B(0,t_0),
\end{cases}
\end{equation}
has a radially symmetric solution $u=u(r)$ where $r=|x|=\sqrt{x_1^2+x_2^2+\cdots+x_n^2}$
    and $Du$  is  a diffeomorphism from $B(0,R_0)$ to $B(0,t_0)$.
\end{lemma}
\begin{proof}
It is easy to calculate that
\begin{align*}
     \frac{\partial r}{\partial x_k}=\frac{x_k}{r},\quad D_k u=u'(r)\frac{x_k}{r},
\end{align*}
and then
\begin{align*}
    \mathrm{div}\left(\frac{Du}{\sqrt{1-|Du|^2}}\right)&=\mathrm{div}\left(\frac{u'(r)}{\sqrt{1-|u'(r)|^2}}\frac{x}{r}\right)
    \\&=\sum_{k=1}^n\left[\left(\frac{u'(r)}{\sqrt{1-|u'(r)|^2}}\right)'\frac{x^2_k}{r^2}+\left(\frac{u'(r)}{\sqrt{1-|u'(r)|^2}}\frac{1}{r}\right)-\frac{u'(r)}{\sqrt{1-|u'(r)|^2}}\frac{{x_k}^2}{r^3}\right]
    \\&=\left(\frac{u'(r)}{\sqrt{1-|u'(r)|^2}}\right)'+\frac{u'(r)}{\sqrt{1-|u'(r)|^2}}\frac{n}{r}-\frac{1}{r}\frac{u'(r)}{\sqrt{1-|u'(r)|^2}}.
\end{align*}
Let $p(r)=\frac{u'(r)}{\sqrt{1-|u'(r)|^2}}$, then we can rewrite the equation \eqref{e6.6.1} as

\begin{equation}\label{e6.6.2}
    p'(r)+\frac{n-1}{r}p(r)=c,
\end{equation}
where $p(0)=u'(0)=0$.

We expand $p(r)$ according to Taylor's formula
\begin{align*}
p(r)=\sum_{k=1}^{+\infty}a_{k}r^k.
\end{align*}

We take the first order derivative of the above equation to get
\begin{align*}
p'(r)=k\sum_{k=1}^{+\infty}a_{k}r^{k-1}.
\end{align*}
Combining with ($\ref{e6.6.2}$), we obtain
\begin{align*}
rp'(r)=k\sum_{k=1}^{+\infty}a_kr^k,
\end{align*}
and then
\begin{equation}\label{e6.6.3}
k\sum_{k=1}^{+\infty}a_kr^k+(n-1)\sum_{k=1}^{+\infty}a_kr^k=r c.
\end{equation}
We compare the coefficients of the terms corresponding to ($\ref{e6.6.3}$)
\begin{align*}
    \begin{cases}
    a_1+(n-1)a_1=c,\;\;k=1,\\
    ka_k+(n-1)a_k=0,\;\; k\geq 2.
\end{cases}
\end{align*}
By calculation, we obtain
\begin{equation}
    \begin{cases}
    a_1=\frac{c}{n},\;\;k=1,\\
    a_k=0,\;\;k\geq 2,
\end{cases}
\end{equation}
and then
\begin{align*}
    p(r)=a_1r=\frac{c}{n}r,\quad
    p'(r)=a_1=\frac{c}{n}.
\end{align*}
Now, we need to solve $u'(r)$,
we know
\begin{equation}
    \frac{u'(r)}{\sqrt{1-|u'(r)|^2}}=\frac{c}{n}r,
\end{equation}
and then
\begin{align*}
    u'(r)=\left(\frac{c^2r^2}{n^2+c^2r^2}\right)^{\frac{1}{2}}.
\end{align*}
The initial value condition is $u'(R_0)=t_0$, so
\begin{align*}
    u'(R_0)=\left(\frac{c^2{R_0}^2}{n^2+c^2{R_0}^2}\right)^{\frac{1}{2}}=t_0,
\end{align*}
and then
\begin{align*}
    c = \frac{n t_0}{\sqrt{1-t_0^2} R_0},
\end{align*}
We can obtain
\begin{equation}\label{e6.6.6}
    u'(r)=\left(\frac{c^2r^2}{n^2+c^2r^2}\right)^{\frac{1}{2}}.
\end{equation}
Integrating from 0 to $r$ for ($\ref{e6.6.6}$), we get
\begin{align*}
    \int_0^ru'(r)dr=\int_0^r\frac{cr}{\sqrt{n^2+c^2r^2}}dr.
\end{align*}
Therefore, we get
\begin{equation}
u(r)=u(0)+\frac{\sqrt{(n^2+c^2r^2)}-n}{c}.
\end{equation}
\end{proof}

By utilizing Lemma \ref{lem6.1}, we can prove the main theorem.

\noindent\textbf{ Proof of Theorem \ref{t1.1}.}

 We know that the existence of solutions to (\ref{e1.1.3}) and (\ref{e1.1.4}) is equivalent to
 the existence of solutions to the following equations which can be written as (\ref{e3.3.4}) in Section 3
\begin{align*}
\begin{cases}
      G(Du,D^2u)=c ,\quad x\in\Omega,\\
  \qquad\;  Du(\Omega)=\tilde{\Omega}.
\end{cases}
\end{align*}

 Let  $\tilde{h}$ and $h$ be the boundary defining functions of $\Omega$ and $\tilde{\Omega}$ constructed in Section 3.
By proposition A.1 in \cite{SM}, we may assume the defining function $\tilde{h}$ and $h$ satisfy the following properties:

(1) $\tilde{h}$ and $h$  are uniformly convex;

(2) $\tilde{h}(x)=0$ for all $x\in \partial \Omega$ and $h(x)=0$ for all $x\in \partial \tilde{\Omega}$;

(3) The sub-level sets of $ \{ x\in \tilde{\Omega}: h(x)\leq t \}$ are balls
when $t$ is sufficiently close to $\inf_{\tilde{\Omega}} h$
and the sub-level sets of $ \{ x\in \Omega: \tilde{h}(x) \leq s \}$ are balls
when $s$ is sufficiently close to $\inf_{\Omega} \tilde{h}$.
 By dividing some positive constants, we may assume that $\inf_{\Omega} \tilde{h} = \inf_{\tilde{\Omega}} h = -1$.

For each $t\in(0,1]$, we consider the sub-level sets of $\tilde{h}$ and $h$:
\begin{equation*}
    \Omega_t :=\{p\in\Omega:\tilde{h}(p)\leq t-1 \}
\end{equation*}
and
\begin{equation*}
    \tilde\Omega_t :=\{q\in\tilde{\Omega}:h(q)\leq t-1\}.
\end{equation*}
Since $\tilde{h}$ and $h$ are uniformly convex, we can see that the sub-level sets $\Omega_t$ and $\tilde{\Omega}_t$ are all uniformly convex domains with smooth boundary.

For each $t\in(0,1]$, we consider the following problem:
\begin{problem}\label{prob2}Find a convex function $u_t: \Omega_t\to \mathbb{R}$ and a bounded function $c(t)$ such that $\nabla u_t$
is a diffeomorphism from $\Omega_t$ to $\tilde{\Omega}_t$ and $G(Du_t(x),D^2u_t(x))=c(t)$ for all $x\in \Omega_t$.
\end{problem}

If $t\in (0,1]$ is sufficiently small, then the sub-level sets $\Omega_t$ and $\tilde{\Omega}_t$ are balls in $\Omega$ and $\tilde{\Omega}$ respectively.
By using Lemma \ref{lem6.1}, we know that
the Problem \ref{prob2} is solvable if $t\in (0,1]$ is sufficiently small.

We define the set
\begin{align*}
    I=\{t\in(0,1]:  \text{Problem \ref{prob2} has at least one solution}\}.
\end{align*}
Therefore the set $I$ is a non-empty.
We claim that $I=(0,1]$, which is equivalent to prove that the set $I$ is not only open, but also closed.
By the invertibility of the
linearized operator in Proposition \ref{e4.4.1} and Theorem 17.6 in \cite{GT}, we know that the set $I$  is an open subset of $(0,1]$.
We next use the a-priori estimates in Section 3 and section 4 to prove that $I$
is a closed subset of $(0, 1]$.
It is equivalent to the fact that for any monotone increasing  sequence $\{t_k \}\subset I$, if $\lim_{k\rightarrow \infty}t_k = t_0$, then $t_0 \in I$.

For each $t_k$, we denote $(u_k,c(t_k))$ solving Problem \ref{prob2}
\begin{equation*}
\begin{cases}
      G(Du_k,D^2u_k)=c(t_k) ,\quad x\in\Omega_{t_k},\\
\qquad\;\;Du(\Omega_{t_k})=\tilde{\Omega}_{t_k}.
\end{cases}
\end{equation*}
Combining Lemma \ref{lem3.6} with Evans-Krylov theory, we can prove that
\begin{equation*}
  \| u_k \|_{C^{2,\alpha}(\bar{\Omega}_{t_k})} \leq C_{38},
\end{equation*}
where $C_{38}$ is independent of $k$.
Since by using Lemma \ref{lem3.6} again, we have $$|c(t_k)|=|G(Du_{t_k},D^2u_{t_k})| \leq C_{38},$$
and hence by using Arzela-Ascoli Theorem, we know that there exists $\tilde{u} \in C^{2,\alpha}(\Omega_{t_0})$,
$\tilde{c} \in \R$ and a subsequence of $\{t_k\}$,
which is still denoted by $\{t_k\}$, such that by letting $k \rightarrow \infty$ to obtain
\begin{equation*}
\begin{cases}
     \| u_k-\tilde{u}\|_{C^2(\Omega_{t_0})}\rightarrow 0,\\
   c(t_k)\rightarrow \tilde{c}.
\end{cases}
\end{equation*}
Lemma \ref{lem3.6} ensure that if $Du_{t_k}$ is a diffeomorphism from $\Omega_{t_k}$ to $\tilde{\Omega}_{t_k}$,
then $D\tilde{u}$ is a diffeomorphism from $\Omega_{t_0}$ to $\tilde{\Omega}_{t_0}$.
Letting $k\rightarrow \infty$, we deduce that
\begin{equation*}
\begin{cases}
      G(D\tilde{u},D^2\tilde{u})=\tilde{c} ,\quad x\in\Omega_{t_0},\\
   \qquad\;\; D\tilde{u}(\Omega_{t_0})=\tilde{\Omega}_{t_0}.
\end{cases}
\end{equation*}
Therefore $t_0 \in I$ and thus $I$ is closed.
Consequently, $I=(0,1]$ and
we complete the proof of Theorem \ref{t1.1}.

{\bf Acknowledgments:} The authors would like to thank  referees for useful
comments, which improve the paper.

\end{document}